\documentclass[11pt]{article}
\usepackage{amsmath}
\usepackage{amsthm}
\usepackage{amssymb}
\usepackage{algorithm}
\usepackage{subfig}
\usepackage{color}
\usepackage[english]{babel}
\usepackage{graphicx}
\usepackage{wrapfig,epsfig}
\usepackage{epstopdf}
\usepackage{url}
\usepackage{graphicx}
\usepackage{color}
\usepackage{epstopdf}
\usepackage{algpseudocode}
\usepackage{scrextend}
\usepackage[T1]{fontenc}
\usepackage{bbm}
\usepackage{comment}

\usepackage{tikz}
\usepackage{hyperref}  
\hypersetup{colorlinks=true,citecolor=blue,linkcolor=blue} 
\usetikzlibrary{arrows}
\usepackage[margin=1in]{geometry}
\graphicspath{{./figs/}}

\author{
  Rasmus Kyng%
\thanks{Supported by ONR grant N00014-17-1-2127.}\\
  \texttt{kyng@seas.harvard.edu}\\
  Harvard University
  \and
  Zhao Song\\
  \texttt{zhaos@seas.harvard.edu}\\
  Harvard University \& UT-Austin
}

\date{}
\title{A Matrix Chernoff Bound for Strongly Rayleigh Distributions and
  Spectral Sparsifiers from a few Random Spanning Trees}
\newtheorem{theorem}{Theorem}[section]
\newtheorem{lemma}[theorem]{Lemma}
\newtheorem{definition}[theorem]{Definition}

\newtheorem{corollary}[theorem]{Corollary}

\newtheorem{fact}[theorem]{Fact}
\newtheorem{remark}[theorem]{Remark}
\newtheorem{claim}[theorem]{Claim}

\newcommand{\abs}[1]{|#1|}

\newcommand{\wh}{\widehat}
\newcommand{\wt}{\widetilde}

\newcommand{\eps}{\epsilon}

\newcommand{\R}{\mathbb{R}}

\newcommand{\norm}[1]{\left\lVert#1\right\rVert}
\renewcommand{\varepsilon}{\epsilon}

\renewcommand{\hat}{\wh}

\DeclareMathOperator*{\E}{{\mathbb{E}}}

\DeclareMathOperator{\supp}{supp}
\DeclareMathOperator{\poly}{poly}

\DeclareMathOperator{\tr}{tr}

\DeclareMathOperator{\wdeg}{wdeg}

\newcommand{\todolow}[1]{\textbf{\color{yellow}[TODO: #1]}}
\renewcommand{\todolow}[1]{}
\newcommand{\Rasmus}[1]{\textbf{\color{red}[Rasmus: #1]}}
\renewcommand{\Rasmus}[1]{}

\makeatletter
\newcommand*{\RN}[1]{\expandafter\@slowromancap\romannumeral #1@}
\makeatother

\usepackage{lineno}

\usepackage{etoolbox}

\makeatletter

\newcommand{\define}[4][ignore]{%
  \ifstrequal{#1}{ignore}{}{
  \@namedef{thmtitle@#2}{#1}}%
  \@namedef{thm@#2}{#4}%
  \@namedef{thmtypen@#2}{lemma}%
  \newtheorem{thmtype@#2}[theorem]{#3}%
  \newtheorem*{thmtypealt@#2}{#3~\ref{#2}}%
}

\newcommand{\state}[1]{%
  \@namedef{curthm}{#1}
  \@ifundefined{thmtitle@#1}{
  \begin{thmtype@#1}
    }{
  \begin{thmtype@#1}[\@nameuse{thmtitle@#1}]
  }
    \label{#1}
    \@nameuse{thm@#1}
  \end{thmtype@#1}
  \@ifundefined{thmdone@#1}{
  \@namedef{thmdone@#1}{stated}%
  }{}
}

\newcommand{\restate}[1]{%
  \@namedef{curthm}{#1}
  \@ifundefined{thmtitle@#1}{
    \begin{thmtypealt@#1}
    }{
  \begin{thmtypealt@#1}[\@nameuse{thmtitle@#1}]
  }
    \@nameuse{thm@#1}
  \end{thmtypealt@#1}
  \@ifundefined{thmdone@#1}{
  \@namedef{thmdone@#1}{stated}%
  }{}
}

\newcommand{\thmlabel}[1]{
  \@ifundefined{thmdone@\@nameuse{curthm}}{\label{#1}
    }{\tag*{\eqref{#1}}}
}
\makeatother

\begin{document}

\begin{titlepage}
  \maketitle
  \begin{abstract}

Strongly Rayleigh distributions are a class of negatively dependent distributions of binary-valued random variables
 [Borcea, Br{\"a}nd{\'e}n, Liggett JAMS 09].
Recently, these distributions have played a crucial role in the analysis of algorithms for fundamental graph problems,
e.g. Traveling Salesman Problem [Gharan, Saberi, Singh FOCS 11].
We prove a new matrix Chernoff bound for Strongly Rayleigh distributions.

As an immediate application, we show that adding together the Laplacians of $\epsilon^{-2} \log^2 n$ random spanning trees gives an $(1\pm \epsilon)$ spectral sparsifiers of graph Laplacians with high probability.
Thus, we positively answer an open question posed in [Baston, Spielman, Srivastava, Teng JACM 13]. Our number of spanning trees for spectral sparsifier matches the number of spanning trees required to obtain a cut sparsifier in [Fung, Hariharan, Harvey, Panigraphi STOC 11].
The previous best result was by naively applying a classical matrix Chernoff bound which requires $\epsilon^{-2} n \log n$ spanning trees.
For the tree averaging procedure to agree with the original graph Laplacian in expectation, each edge of the tree should be reweighted by the inverse of the edge leverage score in the original graph.
We also show that when using this reweighting of the edges, the Laplacian of single random tree is bounded above in the PSD order by the original graph Laplacian times a factor  $\log n$ with high probability, i.e. 
$L_T \preceq O(\log n) L_G$.

We show a lower bound that almost matches our last result, namely that in some graphs, with high probability, the random spanning tree is $\emph{not}$ bounded above in the spectral order by $\frac{\log n}{\log\log n}$ times the original graph Laplacian.
We also show a lower bound that in  $\epsilon^{-2} \log n$ spanning trees are necessary to get a $(1\pm \epsilon)$ spectral sparsifier.






  \end{abstract}
  \thispagestyle{empty}
\end{titlepage}

{\hypersetup{linkcolor=black}
\tableofcontents
}
\newpage




\section{Introduction}

The study of concentration of sums of random variables dates back
to Central Limit Theorems, and hence de Moivre and Laplace
\cite{t04},
while modern concentration bounds for sums of random variables were perhaps first
established by Bernstein \cite{b24}, and a popular variant now known
as Chernoff bounds was introduced by Rubin and published by
Chernoff \cite{c52}.

Concentration of measure for matrix-valued random variables is the
phenomenon that many matrix valued distributions are to close their
mean with high probability, closeness usually being measured by
spectral norm.
Modern quantitative bounds of the form often used in theoretical
computer science were derived by Rudelson \cite{r99},
while Ahlswede and Winter \cite{aw02} established a useful
matrix-version of the Laplace transform that plays a central role in
scalar concentration results such as those of Bernstein.
\cite{aw02} combined this with the Golden-Thompson trace inequality to prove
matrix concentration results. 
Tropp refined this approach, and by replacing the use of
Golden-Thompson with deep a theorem on concavity of certain trace functions due to Lieb, Tropp was able to
recover strong versions of a wide range of scalar concentration
results, including matrix Chernoff bounds, Azuma and Freedman's
inequalities for matrix martingales \cite{t12}.

Matrix concentration results have had an enormous range of applications
in computer science, and are ubiquitous throughout spectral graph
theory \cite{ST04,ss11,ckpprsv17}, sketching \cite{c16}, approximation algorithms \cite{hsss16}, and deep learning \cite{zsjbd17,zsd17}.
Most applications are based on results for independent random
matrices, but more flexible bounds, such as Tropp's Matrix Freedman
Inquality \cite{t11a}, have been used to greatly simplify algorithms,
e.g. for solving Laplacian linear equations~\cite{ks16} and for semi-streaming
graph sparsification \cite{ag09,kpps17}.
Matrix concentration results are also closely related to other popular
tools sampling tools, such as Karger's
techniques for generating sparse graphs that approximately preserve
the cuts of denser graphs \cite{bk96}.

Negative dependence of random variables is an appealing property that
intuition suggests should help with concentration of measure.
Notions of negative dependence can be formalized in many ways.  
Roughly speaking, these notions characterize distributions where 
where some event occurring ensures that other events of interest become
less likely.
A simple example is the distribution of a sequence of coin
flips, conditioned on the total number of heads in the outcome. 
In this distribution, conditioning on some coin coming out heads makes
all other coins less likely to come out heads.
Unfortunately, negative dependence phenomena are not as robust as
positive association which can be established from local conditions
using the powerful FKG theorem~\cite{fkg71}.

Strongly Rayleigh distributions were introduced recently by Borcea,
Br{\"a}nd{\'e}n, and Liggett \cite{bbl09} as a class of negatively
dependent distributions of binary-valued random variables with many
useful properties.
Strongly Rayleigh distributions satisfy useful negative dependence
properties, and retain these properties under natural
conditioning operations.
Strongly Rayleigh distributions also satisfy a powerful stability
property under conditioning known as \emph{Stochastic Covering} \cite{pp14}, which is
useful for analyzing them through martingale techniques.
A measure on $\{0,1\}^n$ is said to be Strongly Rayleigh if its generating
polynomial is real stable \cite{bbl09}.
There are many interesting examples of Strongly Rayleigh
distributions \cite{pp14}: The example mentioned earlier of heads of independent
coin flips conditional on the total number of heads in the outcome;
symmetric exclusion processes; 
determinental point processes and determinental measures on a boolean
lattice.
An example of particular interest to us is the edges of uniform or
weighted random spanning trees, which form a Strongly Rayleigh
distribution.

We prove a Matrix Chernoff bound for the case of $k$-homogeneous
Strongly Rayleigh distributions.
Our bound is slightly weaker than the bound for independent variables.
We give lower bounds that show our bounds are close to tight in some
regimes, but importantly, our lower bounds do not establish separation
from the behaviour of indepedent random matrices, leaving open the
question of whether the true bound should match the independent case
in all regimes -- which seems plausible.
We use our bound to show new concentration results related to random
spanning trees of graphs.
An open question is to find other interesting applications of our
concentration result, e.g. by analyzing concentration for matrices
generated by exclusion processes.

Random spanning trees are one among the most well-studied
probabilistic objects in graph theory, going back to the work of
Kirchoff \cite{k47} in 1847, who gave formula relating the number of
spanning trees in a graph to the determinant of the Laplacian of the
same graph.

Algorithms for sampling of random spanning trees have been studied
extensively,  
~\cite{Guenoche83,
  Broder89, Aldous90,Kulkarni90,Wilson96,
  ColbournMN96,KelnerM09,MadryST15,HarveyX16,dkprs17,dppr17,schild18},
and a random spanning tree can now be sampled in almost linear time \cite{schild18}.

\todolow{rephrase below}
In theoretical computer science, random spanning trees have found a
number of applications, most notably in breakthrough results on
approximating the traveling salesperson problem with
symmetric~\cite{gss11} and asymmetric costs~\cite{AsadpourGMGS10}.
Goyal et al.~\cite{grv09} demonstrated that adding just two random
spanning trees sampled from a bounded degree graph gives a $O(\log n)$
cut sparsifier with probability $1-o(1)$.
Later, it was shown by Fung, Hariharan, Harvey, Panigraphi
\cite{fhhp11}, that if we sample $O(\epsilon^{-2} \log^2 n)$ random spanning
trees from a graph, reweight the tree edges by the inverse of their leverage scores
in the original graph, and average
them together, then whp. we get a graph where every the weight of edges
crossing every cut is approximately
the same in as in the original graph, up to
a factor $(1\pm\epsilon)$.
We refer to this as an
$\epsilon$-cut sparsifier.
The techniques of Fung et~al. unfortunately do not extend to proving
spectral sparsifiers.

Spectral graph sparsifiers were introduced by Spielman and
Teng~\cite{ST04}, who for any graph $G$ showed how to construct a
another graph $H$ with $\epsilon^{-2} n\poly\log n$ edges
s.t. $(1-\epsilon) L_G \preceq L_H \preceq (1+\epsilon)L_G $, which we refer
to as an $\epsilon$-spectral sparsifier.
The construction was refined by Spielman and Srivastava \cite{ss11},
who suggested sampling edges independently\footnote{\cite{ss11}
  analyzed sampling with replacement, but based on \cite{t12}, a 
  folklore result shows the same behavior can be obtained by doing
  independent coin flips for every edge with low leverage score, again
with inclusion probabilities proportional to leverage scores.} with probability proportional to their
\emph{leverage scores}, and brought the number of required
samples down to $\epsilon^{-2} n \log n$.
This analysis is tight in the sense that if fewer than $o(\epsilon^{-2} n \log n)$
samples are used, there will be at least a $1/\poly(n)$ probability of
failure.
Meanwhile, $\epsilon^{-2} n \frac{\log n}{\log \log n}$ independent samples
  in a union of cliques can be shown whp. to fail to give a cut sparsifier.
This can be observed directly from the degree distribution of a single
vertex in the complete graph\todolow{}.
For a variant of \cite{ss11} sampling based on flipping a single coin
for each edge to decide whether to keep it or not, 
it can also be shown that when the expected number of edges is
$\epsilon^{-2} n \frac{\log n}{\log \log n}$,
whp. the procedure fails to give a cut sparsifier.
For arbitrary sparsification schemes, bounds in \cite{bss12} show that 
$\Theta(\epsilon^{-2} n)$ edges are necessary and sufficient to give an
$\epsilon$-spectral sparsifier.

The marginal probability of an edge being present in a random spanning
tree is exactly the leverage score of the edge.
This seems to suggest that combining $\epsilon^{-2} \poly\log n$
spanning trees might give a spectral sparsifier, but the lack of
independence between the sampled edges means the process 
cannot be analyzed using existing techniques.
Observing this, Baston, Spielman, Srivastava, Teng \cite{bsst13}
in their excellent 2013 survey on sparsification noted that
{\it ``it remains to be seen if the union of a small number of random spanning trees can produce a spectral sparsifier.''}
We answer this question in the affirmative. In particular, we show
that adding together $O(\epsilon^{-2} \log^2 n)$ spanning trees with edges
scaled proportional to inverse leverage scores in the original graph
leads to an $\epsilon$-spectral sparsifier.
This matches the bound obtained for cut sparsifiers in \cite{fhhp11}.
Our result also implies their earlier bound since a spectral sparsifier is always a
cut sparsifier with the same approximation quality.
Before our result, only a trivial bound on the number of spanning
trees required to build a spectral sparsifier was known.
In particular standard matrix concentration arguments like those in
\cite{ss11} prove that $O(\epsilon^{-2} n \log n)$ spanning trees
suffice.
Lower bounds in \cite{fhhp11} show that whp.  $\Omega(\log n)$  random
spanning trees are required to give a constant factor spectral
sparsifier.

We show that whp. $\epsilon^{-2} \frac{\log n}{\log \log n}$ random
spanning trees do not give an $\epsilon$-spectral sparsifier.
%
%
We also show that the Laplacian of a single random tree with edges
weighted as above satisfies 
$L_T \preceq O(\log n) L_G$
whp., and we give an almost matching lower bound, showing that in some graphs
whp. $L_T \not\prec \frac{1}{8} \frac{\log n}{\log\log n} L_G$.
Before our work, the main result known about approximating graphs using $O(1)$ random spanning
trees is due to Goyal, Rademacher, Vempala \cite{grv09}, who showed
that surprisingly, when the original graph has bounded degree, 
adding two random spanning trees gives a graph
whose cuts approximate the cuts in the original graph 
up to a factor $O(\log n)$ with
good probability. 
\todolow{good -> high?}
As our result for a single tree establishes only a one-sided bound an
interesting open question remains: Does sampling $O(1)$ random
spanning trees give a $\log n$-factor spectral sparsifier with, say,
constant probability?
\subsection{Previous work}
\paragraph{Chernoff-type bound for matrices.}
Chernoff-like bounds for matrices appear in 
Rudelson \cite{r99} and Ahlswede and Winter \cite{aw02}.
The latter introduced a useful matrix-variant of the Laplace transform
that is central in concentration bounds for scalar-valued matrices.
Their bounds restricted to iid random matrices, an artifact of their
use of the Golden-Thompson inequality for bounding traces.
In contrast, Tropp obtained more flexible concentration bounds for random matrices
by using a result of Lieb to bound the expected trace of various
operators \cite{t12}, including bounds for matrix martingales \cite{t11b}.

In a recent work by Garg, Lee, Song and Srivastava \cite{glss18}, they
show a Chernoff bound for sums of matrix-valued random variables
sampled via a random walk on an expander graph. This work confirms a
conjecture due Wigderson and Xiao. 
The proof of Garg et~al. is also concerned with matrices that are not 
fully independent. In this case 
the matrices are generated from random walks on an expander
graph.
The main idea to deal with dependence issue is using a new
multi-matrix extension of the Golden-Thompson inequality and an
adaptation of Healy's proof of the expander Chernoff bound in the scalar's case \cite{h08} to matrix case.
Their techniques deal with fairly generic types of dependence, and
cannot leverage the very strong stability properties that arise from
the negative dependence and stochastic covering properties of Strongly
Rayleigh distributions.
Harvey and Olver~\cite{ho14} proved a matrix concentration result for
randomized pipage rounding, which can be used to show concentration
results for random spanning trees obtained from pipage rounding, but not
for (weighted) uniformly random spaning trees. The central technical
element of their proof is a new variant of a theorem of Lieb on
concavity of certain matrix trace functions.

Matrix martingales have played a central role in a number of
algorithmic results in theoretical computer science \cite{ks16,cmp16,kpps17}, but beyond a reliance on Tropp's 
Matrix Freedman Inequality, these works have little in common with our approach. 
However, our bound does share a technical similarity with \cite{ks16},
namely that a sequence of increasingly restricted random choices in a
martingale process lead to a $\log n$ factor in a variance bound. 

\paragraph{Strongly Rayleigh Distributions in Theoretical Computer Science.}
Perhaps the most prominent result on Strongly Rayleigh distributions
in theoretical computer science is the generalization of
\cite{mms13} to Strongly Rayleigh distributions.

The central technical result of \cite{mms13} essentially shows that
given a collection of independent random vectors $v_1,\ldots, v_m$
with finite support in $\mathbb{C}^n$ s.t. $\sum_{i=1}^m \E
[v_iv_i^*] = I$ and for all $i$, $\norm{v_i}^2 \leq \eps$,
then ${\Pr[\norm{\sum_{i=1}^m v_iv_i^*} \leq
  (1+\sqrt{\eps})^2]} >0$.
\cite{ao15} establishes a related result for $k$-homogeneous Strongly
Rayleigh distributions, though they require an additional constraint
on the marginal probability that any given random variable is
non-zero being bounded above by $\delta$, and then establish 
${\Pr[\norm{\sum_{i=1}^m v_iv_i^*} \leq  4(\epsilon + \delta)
  + 2(\epsilon+\delta)^2]} >0$.
Based on this, \cite{ao15} shows\footnote{Their full statement is more
  general, see \cite{ao15} Corollary 1.9.} that given an unweighted $k$-edge
connected graph $G$ where every
edge has leverage score at most $\epsilon$, there exists an unweighted
spanning tree s.t. $L_T \preceq O(\frac{1}{k} + \epsilon) \cdot
L_G$. This is referred to as a spectrally thin tree with parameter $ O(\frac{1}{k} + \epsilon) $.

\cite{agr16} showed how to algorithmically sample from $k$-homogeneous
Determinental Point Process in time $\poly(k) n
\log(n/\epsilon)$, where $n$ is the dimension of matrix giving rise to
the determinental point process and $\epsilon$ is the allowed total
variation distance.
Their techniques are based on generalization proofs of expansion in
the base graph associated with a balanced matroid, a result first
established by \cite{fm92}.

\paragraph{Random spanning trees.}
Algorithms for sampling random spanning trees have a long history, but
only recently have they explicitly used matrix concentration
\cite{dkprs17,dppr17,schild18}.
The matrix concentration arguments in these papers, however, deal
mostly with how modifying a graph results in changes to the
distribution of random spanning trees in the graph.
We instead study how closely random spanning trees resemble the graph
they were initially sampled from. Whether our result in turn has
applications for improving sampling algorithms for random spanning
trees is unclear. 

The fact that spanning tree edges exhibit negative dependence has been
used strikingly in concentration arguments by Goyal et
al. \cite{grv09} to show that two random spanning trees gives  $O(\log
n)$-factor approximate cut sparsifier in bounded degree graphs, with
good probability.
\todolow{good -> high?}
This is clearly false when sampling the same number of edges
independently, because this graph has large probability of having
isolated vertices.\todolow{confirm}
Goyal et~al. improve over independent sampling by leveraging the fact
that for a fixed tree, in some sense, very few cuts of a given
size exist.
This is a variant of Karger's famous cut-counting techniques \cite{k93,ks96}
specialized to unweighted trees.


Uses of negatively dependent Chernoff bounds applied to tree
edges also appeared in works on approximation algorithms for TSP
problems \cite{gss11, AsadpourGMGS10}, where additionally the
connectivity properties of the tree play an important role
\todolow{does that make sense?}

In contrast, the techniques of Fung et~al. \cite{fhhp11} show that 
$O(\epsilon^{-2} \log^2 n)$ spanning trees suffice to give a 
$(1\pm\epsilon)$-cut sparsifier, but they 
do not show that tree-based
sparsifiers improve over independent sampling,
The focus of their paper is to establish that wide range of different
techniques for choosing sampling probabilities all give cut
sparsifiers, by establishing a more flexible framework than the
original cut-sparsifier results of Benczur-Karger
\cite{bk96}, using related cut-counting techniques (see
\cite{k93,ks96}).
To extend their results to spanning trees, they
simply observe that the (scalar-valued) Chernoff bounds they use
directly apply to negatively dependent variables, and hence edges
in spanning trees.

Fung et~at. \cite{fhhp11} also establish a lower bound, showing that for
any constant $c$, there exists a graph for which obtaining a
factor $c$-cut sparsifier by averaging trees requires using at least
$\Omega(\log n)$ trees to succeed with constant probability.
%





\subsection{Our results and techniques}

\define{thm:srmatchernoff}{Theorem}{{\rm (First main result, a Matrix Chernoff Bound
  $k$-homogeneous Strongly Rayleigh Distributions)}{\bf .} 
Suppose $(\xi_1, \ldots, \xi_m) \in  \{ 0,1\}^m$ is a
random vector of $\{0,1\}$ variables whose distribution is
$k$-homogeneous and Strongly Rayleigh.

Given a collection of PSD matrices $A_1,\ldots A_m \in \mathbb{R}^{n
  \times n}$
s.t. for all $e \in [m]$ we have $\norm{A_e} \leq R$ and 
$\norm{\E[ \sum_{e} \xi_e A_e]} \leq \mu$.

Then for any $ \epsilon > 0$,
\ifdefined\focsversion
\begin{align*}
& ~ \Pr\left[ 
\norm{\sum_{e} \xi_e A_e - \E\left[ \sum_{e} \xi_e
      A_e\right]} \geq \epsilon \mu
\right] \\
\leq & ~ 
n \exp\left( - \frac{\epsilon^2 \mu }{R (\log k + \epsilon) } \Theta(1) \right)  
\end{align*}
\else
\begin{align*}
\Pr\left[ 
\norm{\sum_{e} \xi_e A_e - \E\left[ \sum_{e} \xi_e
      A_e\right]} \geq \epsilon \mu
\right]
\leq 
n \exp\left( - \frac{\epsilon^2 \mu }{R (\log k + \epsilon) } \Theta(1) \right)  
\end{align*}
\fi
}

\state{thm:srmatchernoff}

This Matrix Chernoff bound matches the bounds due to Tropp~\cite{t12}, 
up to the $\log k$ factor in the exponent.
Our lower bounds rule out that a much stronger bound is true
in the $\epsilon \approx \log k$ regime, since the level of
concentration we prove at this $\epsilon$ is stronger than what is ruled
out for $\epsilon \approx \frac{\log k}{\log \log k}$ by the lower
bound in Theorem~\ref{thm:intro_one_spanning_tree_upper}.
However, this does \emph{not} rule out that the $\log k$ factor in our
bound is unnecessary.
I.e. we cannot rule out that a stronger concentration statement might
match the bound for the independent case given by Tropp~\cite{t12}.

\begin{remark} 
When the Strongly Rayleigh distribution is in fact a product
distribution on a collection of $l$ Strongly
Rayleigh distributions that are each $t$-homogeneous, then the joint distribution is $l t$-homogeneous Strongly
Rayleigh, but in Theorem~\ref{thm:srmatchernoff}, the factor
$\log(lt)$ can be replaced by a factor $\log t$.
This applies to the case of independent random spanning trees, but
only gives a constant factor improvement in the number of trees required.
\end{remark}

Our work is related to the concentration inequality of
Peres and Pemantle~\cite{pp14}, who showed a concentration result for
scalar-valued Lipschitz functions of Strongly Rayleigh distributions.
They used Doob martingales (martingales constructed from sequences of
conditional expectations) to prove their result.
We use a similar approach for matrices, constructing Doob matrix
martingales from our Strongly Rayleigh distributions.
In addition, we use the stochastic covering property of Strongly
Rayleigh distributions observed by Peres
and Pemantle, but implicitly derived in \cite{bbl09}.
This property leads to bounded differences in Doob martingale
sequences for scalars. 
As in the scalar setting, it is possible to show 
concentration results for matrix-valued martingales.
We use the Matrix Freedman inequality\footnote{Note, however, that we are able to prove deterministic bounds on the
\emph{predictable quadratic variation process}, which means the bound
we use is more analogous to a matrix version
Bernstein's inequality, adapted to martingales.
We resort to the more complicated Freedman's inequality
only because it gives a directly applicable statement that is known in the literature.
}
 of Tropp.
This inequality allows makes it possible to 
establish strong concentration bounds based on control of sample
norms and control of the \emph{predictable quadratic variation process}
of the martingale, a matrix-valued object that is used to
measure variance (see \cite{t11b}). 
We show that as in the scalar setting, the stochastic covering
property of Strongly Rayleigh distributions leads to bounded
differences for Doob matrix martingales.
But, we also combine the stochastic covering property with 
deceptively simple matrix martingale properties and a negative
dependence condition to derive additional bounds on the predictable quadratic variation
process of the martingale.
The key negative dependence property we use is a simple observation that
generalizes, to $k$-homogeneous Strongly Rayleigh distributions,
 the fact that in a random spanning tree, conditioning on
the presence of a set of edges lowers the marginal probability of every
other graph edge being in the tree (see Lemma~\ref{lem:intro_shrinking_marginals}).
While we frame it differently, it is essentially an immediate consequence
of statements in \cite{bbl09}.
The surprise here is how useful this simple observation is for
removing issues with characterizing conditional $k$-homogeneous
Strongly Rayleigh distributions.
As a corollary we get our second main result.
\define{thm:intro_sum_spanning_trees}{Theorem}{{\rm(Second main result, concentration bound of a batch of independent random spanning trees)}{\bf.}
Given as input a weighted graph $G$ with $n$ vertices and a parameter $\epsilon > 0$, let $T_1, T_2, \cdots, T_t$ denote $t$ independent inverse leverage score weighted random spanning trees,  if we choose $t = O(\epsilon^{-2} \log^2 n)$ then with probability $1-1/\poly(n)$,
\begin{align*}
(1-\epsilon)  L_G \preceq \frac{1}{t} \sum_{i=1}^t L_{T_i} \preceq (1+\epsilon) L_G.
\end{align*}
}
\state{thm:intro_sum_spanning_trees}

Prior to our work, only a trivial bound on the number of spanning
trees required to build a spectral sparsifier was known, namely that standard matrix concentration arguments like those in
\cite{ss11} prove that $O(\epsilon^{-2} n \log n)$ spanning trees suffice.
Note that, the number of spanning trees required to build spectral
sparsifier in our Theorem~\ref{thm:intro_sum_spanning_trees} matches
the number of spanning trees required to construct cut sparsifier in
previous best result \cite{fhhp11}.
The total edge count we require is
$\Theta(\epsilon^{-2} n \log^2 n)$, worse by a factor
$\log n$ than the bound for independent edge sampling obtained in $\cite{ss11}$.
It is not clear whether this factor in necessary.

\begin{remark}
  Suppose we apply our Theorem~\ref{thm:intro_one_spanning_tree_upper}
  to show that any single random spanning tree satisfies  
  $L_T \preceq O(\log n) \cdot L_G$ whp. This is tight up a $\log\log
  n$ factor.
  Then, one can from use this to derive
  Theorem~\ref{thm:intro_sum_spanning_trees} based on a 
  standard (and tight) Matrix Chernoff bound, and a
  (laborious) combination of Doob martingales and stopping time
  arguments similar to those found in \cite{t12,ks16}.
  This line of reasoning will lead to the same bounds as
  Theorem~\ref{thm:intro_one_spanning_tree_upper}.
  Thus, unless one proves more than just a norm bound for each
  individual tree,
  it is not possible improve over our result, except for $\log \log n$ factors.
\end{remark}
Like the work of Fung et~al. our results for spanning trees do not improve over
the independent case.
Fung et~al. achieved their result by combining cut
counting techniques with Chernoff bounds for scalar-valued negatively dependent variables.
In our random matrix setting,
there are no clear candidates for a Chernoff bound for negatively
dependent random matrices that we can adopt, and this type of bound is exactly what
we develop in the Strongly Rayleigh case.

We establish a one-sided concentration result for a single tree, namely that
whp. $L_T \preceq O(\log n) \cdot L_G$.
Again, this is a direct application of Theorem~\ref{thm:srmatchernoff}.
\define{thm:intro_one_spanning_tree_upper}{Theorem}{{\rm(Third main result, upper bound for the concentration of one random spanning tree)}{\bf.}
Given a graph $G$, let $T$ be a random spanning tree, then with probability at least $1-1/\poly(n)$
\begin{align*}
L_T \preceq O(\log n) \cdot L_G.
\end{align*}
}
\state{thm:intro_one_spanning_tree_upper}
This upper bound is tight up to a factor $\log\log n$ as shown by our
almost matching lower bound stated below.
\define{thm:intro_one_spanning_tree_lower_high_probability}{Theorem}{{\rm(Lower bound for the concentration of one random spanning tree)}{\bf.}
  For any $n \geq 2^{2^6}$, 
 there is an unweighted graph $G$ with $n$ nodes, s.t. if we sample an inverse leverage score weighted random spanning tree $T$, then with probability at least $1 - e^{-n^{.4}}$,
\begin{align*}
L_T \not\prec  \frac{ \log n }{ 8 \log \log n} \cdot L_G.
\end{align*}
}
\state{thm:intro_one_spanning_tree_lower_high_probability}

Trivially, the presence of degree one nodes in $L_T$ means
that in a complete graph, $L_G \not\prec L_T $.
So choosing any other scaling of the tree will make at least one of
the inequalities $L_T \not\prec \frac{1}{8} \log n / \log \log n \cdot L_G$
and $L_G \not\prec  L_T $ true with a larger gap.
Note that in the complete unweighted graph the trees we consider have weight
$\Theta(n)$ on each edge. A random spanning tree in the complete graph has diameter about
$\sqrt{n}$ \cite{rs67}.
This can be shown to imply that for an \emph{unweighted}
random tree $\hat{T}$, $L_G \not\prec
\sqrt{n} L_{\hat{T}}.$
But once we scale up every edge of the tree by a factor $\Theta(n)$,
the diameter bound no longer directly implies a spectral gap of the
form
$L_G \not\prec
\alpha L_{T}$ for some $\alpha$.
In a ring graph, we get $L_G \not\prec (n-2) L_T $ and $L_T \not\prec L_G$.

We can also show in general that $L_T \not\prec 10 \log n \cdot L_G$, but with a much smaller probability.
\define{thm:intro_one_spanning_tree_lower_low_probability}{Theorem}{{\rm(Lower bound for the concentration of one random spanning tree)}{\bf.}
For any $n \geq 4$, there is an unweighted graph $G$ with $n$ nodes, s.t. if we sample an inverse leverage score weighted random spanning tree $T$, 
 then with probability at least $2^{- 150\log n \log \log n}$,
\begin{align*}
L_T \not\prec 10 \log n \cdot L_G.
\end{align*}
}
\state{thm:intro_one_spanning_tree_lower_low_probability}
%
%
And we show a lower bound for $\epsilon$-spectral sparsifiers for
random spanning trees.
\define{thm:intro_multiple_spanning_tree_lower_high_probability}{Theorem}{{\rm(Lower bound for the concentration of multiple random spanning trees)}{\bf.}
For any $n \geq 2^{100}$, there is an unweighted graph $G$ with $n$ nodes, s.t. for any accuracy parameter $\epsilon \in ( 5 n^{-0.1} ,1/2)$, if we sample $t = 0.05 \epsilon^{-2} \log n$ independent random spanning trees with edges weighted
by inverse leverage score,
then with probability at least $1 - e^{-n^{.39}}$,
\begin{align*}
(1-\epsilon)  L_G \not\prec \frac{1}{t} \sum_{i=1}^t L_{T_i}
\text{ and }
\frac{1}{t} \sum_{i=1}^t L_{T_i} \not\prec (1+\epsilon) L_G.
\end{align*}
}
\state{thm:intro_multiple_spanning_tree_lower_high_probability}

Our lower bound is incomparable with that of Fung et~al.\cite{fhhp11}, who showed that for
any constant $c$, there exists a graph for obtaining a
factor $c$-cut sparsifier by averaging trees requires using at least
$\Omega(\log n)$ trees to succeed with constant probability.
%
 Where Fung et.~al \cite{fhhp11}  used triangles in their lower bound construction,
 our bad examples are based on
collections of small cliques, which lets us ensure cut differences
in even a single tree, by giving longer-tailed degree
distributions.
All of our lower bounds are based on simple constructions from
collections of edge disjoint cliques, and use the fact that the exact
distribution of degrees of a fixed vertex in a random spanning tree of the complete
graph is known.
Note that a lower bound for cut approximation implies a lower-bound for
spectral approximation, because the contrapositive statement is true: 
spectral approximation implies cut approximation.

\begin{remark}
  In fact, all our lower bounds also directly apply for cut approximation, which is a strictly stronger result.
  For example, there is an unweighted graph $G$, s.t. if we sample an inverse leverage score weighted random spanning tree $T$,
  then with probability at least $1 - e^{-n^{.4}}$,
  $T$ has a cut which is larger than the corresponding cut in $G$ by a
  factor  
$\frac{ \log n }{ 8 \log \log n} $.
\end{remark}

\paragraph{Connection to Spectrally Thin Trees}
Using their MSS-type existence proof for ``small norm outcomes'' of 
homogeneous Strongly Rayleigh distributions, \cite{ao15} showed that in an unweighted $k$-edge
connected graph $G$ where every
edge has leverage score at most $\epsilon$, there exists an unweighted
spanning tree $\hat{T}$ s.t. $L_{\hat{T}} \preceq O(\frac{1}{k} + \epsilon) \cdot
L_G$. This is referred to a spectrally thin tree with parameter $ O(\frac{1}{k} + \epsilon) $.

In contrast, applying our $k$-homogeneous Strongly Rayleigh Matrix Chernoff bound to
an unweighted graph $G$ where every edge has leverage score at most $\epsilon$,
we can show that an unweighted \emph{random} spanning tree satisfies 
$L_T \preceq O( \epsilon \log n )  L_G$ with high probability.
This follows immediately from our Theorem~\ref{thm:intro_one_spanning_tree_upper},
because if we let $T$ denote the unweighted spanning tree and
$\hat{T}$ corresponding spanning tree with edges weighted by inverse
leverage scores, then
\begin{align*}
L_{\hat{T}}  = & ~ \sum_{e \in T}  w(e) b_e b_e^\intercal \\
\preceq & ~ \epsilon \sum_{e \in T} \frac{1}{l(e)} w(e) b_e b_e^\intercal  \\
= & ~ \epsilon L_{\hat{T}} \\
\preceq & ~ O( \epsilon \log n ) L_G
\end{align*}
whp.

The proof in \cite{ao15} is based on an adaptation of the \cite{mms13} 
proof, and does not have clear parallels with our approach.
Whereas the key properties of Strongly Rayleigh distributions that we
use are stochastic covering (a property that limits change in
a distribution under conditioning) and conditional negative
dependence, the central element of their approach is a proof that
certain mixed characteristic polynomials associated with
$k$-homogeneous distributions are real stable when the original
distribution is Strongly Rayleigh.

The following Lemma captures a simple but crucial property of Strongly
Rayleigh distributions.

\begin{lemma}[Shrinking Marginals]\label{lem:intro_shrinking_marginals}
Suppose $ (\xi_1, \ldots, \xi_m) \in  \{ 0,1\}^m$ is a
random vector of $\{0,1\}$ variables whose distribution is
$k$-homogeneous and Strongly Rayleigh,
then any set $S \subseteq [m]$ with $|S| \leq k$
for all $j \in [m]\setminus S$
\[
\Pr[ \xi_j = 1 | \xi_S = \mathbf{1}_S]
\leq
\Pr[ \xi_j = 1 ] .
\]
\end{lemma}
\ifdefined\focsversion
We provide the proof of this lemma in the full version.
In the remaining sections, we prove
Theorems~\ref{thm:srmatchernoff}, 
\ref{thm:intro_sum_spanning_trees}
and \ref{thm:intro_one_spanning_tree_upper}.
We defer the proofs of our other results to the full version.
\else
We provide a proof in Section~\ref{sec:marginals}.
\fi








\section{Notation}
We use $[n]$ to denote set $\{1,2,\cdots, n\}$.
Given a vector $x$, we use $\norm{x}_0$ to denote the number of
non-zero entries in the vector.

\paragraph{Matrix and norms.} 
For a matrix $A$, we use $A^\top$ to denote the transpose of $A$. We say
matrix $A$ is positive semi-definite (PSD) if $A = A^\top$ and $x^\top
A x \geq 0$ for all $x \in \R^n$. We use $\succeq, \preceq$ to denote
the semidefinite ordering, e.g. $A \succeq 0$ denotes that $A$ is PSD,
and  $A \succeq B$ means $A-B \succeq 0$.
We say
matrix $A$ is positive definite (PD) if $A = A^\top$ and 
$x^\top A x >0$ for all $x \in \R^n-\{0\}$.
$A \succ B$  means $A-B$ is PD.

For matrix $A \in \R^{n \times n}$, we define $\| A \|$ to be the spectral norm of $A$, i.e.,
\begin{align*}
\| A \| = \max_{ \| x \|_2 = 1, x \in \R^n } x^\top A x.
\end{align*}

Let $\tr(A)$ denote the trace of a square matrix $A$. We use $\lambda_{\max}(A)$ to denote the largest eigenvalue of matrix $A$. For symmetric matrix $A \in \R^{n \times n}$, $\lambda_{\max}(A) = \| A \|$ and $\tr(A) \leq n \| A \|$.



\paragraph{The Laplacian matrix-related definitions.}
Let $G=(V,E,w)$ be a connected weighted undirected graph with $n$ vertices and $m$ edges and edge weights $w_e > 0 $. If we orient the edges of $G$ arbitrarily, we can write its Laplacian as $L= B^\top W B$, where $B \in \R^{m \times n}$ is the signed edge-vertex incidence matrix and defined as follows
\begin{align*}
B(e,v) = \begin{cases}
1, & \text{~if~}v\text{~is~}e\text{'s~head};\\
-1, & \text{~if~}v\text{~is~}e\text{'s~tail};\\
0, & \text{otherwise.} 
\end{cases}
\end{align*}
and $W \in \R^{m\times m}$ is the diagonal matrix with $W(e,e) = w_e$.

\section{Preliminaries}

\subsection{Useful facts and tools}
This section, we provide some useful tools.
\ifdefined\focsversion
For completeness, we prove the following statement in the full version.
\else
For completeness, we prove the following statement in Appendix~\ref{sec:appendix}.
\fi
\begin{fact}\label{fac:symmetric_matrix_triangle_inequality}
For any two symmetric matrices
\begin{align*}
(A - B)^2 \preceq 2A^2 + 2B^2.
\end{align*}
\end{fact}

\subsection{Strongly Rayleigh distributions}
This section provides definitions related to Strongly Rayleigh distributions. For more details, we refer the readers to \cite{bbl09,pp14}.

Let $\mu: 2^{[n]} \rightarrow \R_{\geq 0}$ denote a probability distribution over $2^{[n]}$, and $\sum_{S \subseteq [n]} \mu(S) = 1$. 

Let $x_1, x_2, \cdots, x_n$ denote $n$ variables, we use $x$ to denote $(x_1, x_2, \cdots, x_n)$. For each set $S \subseteq [n]$, we define $x_S = \prod_{i \in S} x_i$. We define the {\bf generating polynomial} for $\mu$ as follows
\begin{align*}
f_\mu( x ) = \sum_{S \subseteq [n]} \mu(S) \cdot x_S.
\end{align*}

We say distribution $\mu$ is $k$-homogeneous if the polynomial
$f_{\mu}$ is a homogeneous polynomial of degree $k$. In other words,
for each $S \in \supp(\mu)$, $|S|=k$.

We say a polynomial $p(x_1, x_2, \cdots, x_n)$ is stable, if
$\text{Im}(x_i) > 0, \forall i \in [n]$, then $p(x_1,\cdots, x_n) \neq
0$. We say polynomial $p$ is real stable, it is stable and all of its
coefficients are real. We say $\mu$ is a {\bf Strongly Rayleigh}
distribution if $f_{\mu}$ is a real stable polynomial.

\begin{fact}[Conditioning on subset of coordinates]
\label{fac:srcondclosure}
Consider a random vector $(\xi_1, \ldots, \xi_m) \in  \{ 0,1\}^m$
 whose distribution is $k$-homogeneous Strongly Rayleigh.
Suppose we get a binary vector  $\mathbf{b} =  (b_1, \ldots, b_t) \in  \{
0,1\}^t$ with $\norm{\mathbf{b}}_0 = l \leq k$, and we get a set $S \subset [m]$ with $\abs{S} = t$.
Then conditional on $ \mathbf{\xi}_{S} = \mathbf{b}$, 
the distribution of $ \mathbf{\xi}_{ [m] \setminus S}$
is $(k-l)$-homogeneous Strongly Rayleigh.
\end{fact}

This fact tells us that if we condition on the value of some entries
in the vector, the remaining coordinates still have a Strongly
Rayleigh distribution.

\begin{fact}[Stochastic Covering Property]
\label{fac:scp}
Consider a random vector $(\xi_1, \ldots, \xi_m) \in  \{ 0,1\}^m$
 whose distribution is $k$-homogeneous Strongly Rayleigh.
Suppose we are given an index $i \in [m]$.
Let $\xi' = \xi_{[m] \backslash \{i\}}$ be the distribution on entries
of $\xi$ except $i$.
Let $\xi''$ be the distribution of $\xi_{[m] \backslash \{i\}}$
\emph{conditional} on $\xi_{i} = 1$.
Then, there exists a coupling between $\xi'$ and $\xi''$ (i.e. a joint
distribution the two vectors), s.t. in every outcome of the coupling 
the value of $\xi'$ can be obtained from the value of $\xi''$ by
either changing a single from $0$ to $1$ or by leaving all entries unchanged.
\end{fact}
This fact is known as the \emph{Stochastic Covering Property} (see
\cite{pp14}). It gives us a convenient tool for relating the
conditional distribution of a subset of the coordinates of the vector
to the unconditional distribution.

Note that by Fact \ref{fac:srcondclosure}, the distribution of $\xi''$
used in Fact~\ref{fac:scp} is $k-1$ homogeneous.
In contrast, the outcomes of $\xi'$ may have $k$ or $k-1$ ones.
Fact~\ref{fac:scp} tells us that we can pair up all the outcomes of
the conditional distribution $\xi''$ with outcome of the unconditional
distribution $\xi'$ s.t. only a small change is required to make them equal.
This tells us that the distribution is in some sense not changing too
quickly under conditioning.

\subsection{Random spanning trees}
\Rasmus{Need better explanation and reword English a bit.}
We provide the formal definition of random spanning tree in this section.

We use the same definitions about spanning trees as \cite{dkprs17}. Let ${\cal T}_G$ denote the set of all spanning subtrees of $G$. We now define a probability distribution on these trees.
\begin{definition}[$w$-uniform distribution on trees]
\label{def:wtUnifTreeDistr}
Let ${\cal D}_G$ be a probability distribution on ${\cal T}_G$ such that
\begin{align*}
\Pr_{X \sim {\cal D}_G }[ X = T ] \propto \prod_{e \in T} w_e.
\end{align*}
\end{definition}
We refer to ${\cal D}_G$ as the $w$-uniform distribution on ${\cal
  T}_G$. When the graph $G$ is unweighted, this corresponds to the
uniform distribution on ${\cal T}_G$.
Crucially, random spanning tree distributions are Strongly Rayleigh,
as shown in \cite{bbl09}.
\begin{fact}[Spanning Trees are Strongly Rayleigh]
\label{fac:spantreemargin}
In a connected weighted graph $G$, the $w$-uniform distribution on
spanning trees is $(n-1)$-homogeneous Strongly Rayleigh.
\end{fact}

\begin{definition}[Effective Resistance]
The effective resistance of a pair of vertices $u,v \in V_G$ is defined as 
\begin{align*}
R_{\mathrm{eff}} (u,v) = b^\top_{u,v} L^\dagger b_{u,v},
\end{align*}
where $b_{u,v}$ is an all zero vector corresponding to $V_G$, except for entries of $1$ at $u$ and $-1$ at $v$.
\end{definition}

The a reference for following standard fact about random spanning
trees can be found in \cite{dkprs17}. 

\begin{definition}[Leverage Score]
The statistical leverage score, which we will abbreviate to leverage score, of an edge $e = (u,v) \in E_G$ is defined as 
\begin{align*}
l_e = w_e R_{\mathrm{eff}}(u,v).
\end{align*}
\end{definition}

\begin{fact}[Spanning Tree Marginals]
\label{fac:spantreemargin}
The probability $\Pr[e]$ that an edge $e\in E_G$ appears in a tree sampled $w$-uniformly randomly from ${\cal T}_G$ is given by
\begin{align*}
\Pr[e] = l_e,
\end{align*}
where $l_e$ is the leverage score of the edge $e$.
\end{fact}


\section{A Matrix Chernoff Bound for Strongly Rayleigh Distributions}

We first define a mapping which maps an element into a psd matrix.
\begin{definition}[$Y$-operator]
We use $\Gamma$ to denote $[m]$, 
we define a mapping $Y : \Gamma \rightarrow \R^{n \times n}$ such that $Y_{e}$ is a psd matrix and $\| Y_e \| \leq R$.
\end{definition}


Throughout this section, we will use $\xi \in \{0,1\}^m$ to denote a
random length $m$ boolean vector whose distribution is $k$-homogeneous
Strongly Rayleigh.
For any set $S \subseteq [m]$, we use $\xi_S$ to denote the length
$|S|$ vector that only chooses the entry from indices in $S$.

We will frequently need to work with a different representation of the
random variable $\xi$.
We use $\gamma$ to denote this second representation.
The random variable $\gamma$ is composed of a sequence of $k$ random
indices $\gamma_1, \gamma_2,
\cdots, \gamma_k$, each of which takes a value $e_1, e_2, \cdots, e_k
\in [m]$.
The indices give the locations of the ones in $\xi$, i.e. in an
outcome of the two variables $(\xi,\gamma)$, we always have 
$\xi_{\{ \gamma_1, \gamma_2,
\cdots, \gamma_k \}} =  {\bf 1} $ and 
$\xi_{[m] \setminus \{ \gamma_1, \gamma_2,
\cdots, \gamma_k \}} =  {\bf 0} $.
Additionally, we want to ensure that the distribution of $\gamma$ is
invariant under permutation: This can clearly be achieved by 
starting with any distribution for $\gamma$ that satisfies the coupling
with $\xi$ and the applying a uniformly random permutation to reorder
the $k$ indices of $\gamma$ (see \cite{pp14} for a further discussion).

For convenience, for each $i\in [k]$, we define $\gamma_{\leq i}$ and
$\gamma_{\geq i}$ as abbreviated notation for 
\begin{align*}
\gamma_{1}, \gamma_{2}, \cdots, \gamma_{i}  \text{~and~} \gamma_{i}, \gamma_{i+1}, \cdots, \gamma_{k}
\end{align*}
respectively.
Let $S = \{ e_1, e_2, \cdots, e_i\}  \subset [m]$ be one possible assignment for
indices of a subset of the ones in $\xi$, (we require $i \leq k$).
Then the distribution of $\xi_{[m]\setminus S}$ conditional on
$\xi_{S} = {\bf 1}$ is the same as the the distribution of
$\xi_{[m]\setminus S}$ conditional on
$(\gamma_1, \gamma_2, \cdots, \gamma_i) =  (e_1, e_2, \cdots, e_i) $.
In other words, in terms of the resulting distribution of
$\xi_{[m]\setminus S}$, it is equivalent to condition on either $\gamma_{\leq
  i}$ or $\xi_{ \gamma_{ \leq i } } = {\bf 1}$.
We define matrix $Z \in \R^{n \times n}$ as follows.
\begin{definition}[$Z$]
Let $Z$ denote $\sum_{e \in \Gamma} \xi_e \cdot Y_e$ where $\| \xi \|_0 = k$. Due to the relationship between $\xi$ and $\gamma$, we can also write $Z$ as 
\begin{align*}
Z = \sum_{i=1}^k Y_{\gamma_i}.
\end{align*}
For simplicity, for each $i\in [k]$, we define $Z_{\leq i}$ and $Z_{\geq i}$ as follows,
\begin{align*}
Z_{\leq i} = \sum_{j=1}^{i} Y_{\gamma_j} \mathrm{~and~} Z_{\geq i} = \sum_{j=i}^{k} Y_{\gamma_j}.
\end{align*}
\end{definition}

We define a series of matrices $M_i \in \R^{n \times n}$ as follows
\begin{definition}[$M_i$, martingale]\label{def:M}
We define $M_0 = \E[ Z ]$. For each $i \in \{1,2,\cdots,k-1\}$, we define $M_i$ as follows
\begin{align*}
M_i = \E_{\gamma_{\geq i+1}} [ Z ~|~ \gamma_1, \cdots, \gamma_i].
\end{align*}
\end{definition}
It is easy to see that 
\begin{align*}
\E_{\gamma_{i+1}} [ M_{i+1} ] = M_i,
\end{align*}
which implies
\begin{align}
\label{eq:mdiffzeromean}
\E_{\gamma_{i+1}} [ M_{i+1} - M_i ~|~ \gamma_1, \cdots, \gamma_i ] = 0.
\end{align}
Note that we can split $Z$ up as
\begin{align}\label{eq:Z_in_M_i_plus_1}
Z = & ~ \sum_{j=i+2}^{ k } Y_{ {\gamma}_{j}} + Y_{ {\gamma}_{i+1}} + \sum_{j=1}^{i} Y_{\gamma_j} \notag \\
= & ~  {Z}_{\geq i+2} + Y_{ {\gamma}_{i+1}} + Z_{\leq i}
\end{align}
And similarly $Z={Z}_{\geq i+1} + Z_{\leq i} $.

In order to relate $M_{i}$ and $M_{i+1}$, we will consider a fresh
copy of  $\gamma_{\geq i+1}$ which we denote by $\wh{\gamma}_{\geq i+1}$.
We denote the corresponding fresh copy of $Z_{\geq i +1}$, by
$\wh{Z}_{\geq i+1}$.
We can now give an equivalent definition of  $M_{i}$ in terms of the expectation over
$\wh{\gamma}_{\geq i+1}$,
while  $M_{i+1}$ is still defined in terms of the expectation over
${\gamma}_{\geq i+2}$, so that 
\begin{align}
\label{eq:mdecomps}
M_i = \E_{\wh{\gamma}_{\geq i+1}} [ \wh{Z}_{\geq i+1} + Z_{\leq i} ~|~ \gamma_1, \cdots, \gamma_i]
\text{ and } 
M_{i+1} = \E_{{\gamma}_{\geq i+2}} [ 
 {Z}_{\geq i+2} + Y_{ {\gamma}_{i+1}} + Z_{\leq i}
 ~|~ \gamma_1, \cdots, \gamma_i,  \gamma_{i+1} ]
\end{align}
Note that both still depend on the same $\gamma_{\leq i}$ vector, and
$M_{i+1}$ depends on $\gamma_{i+1}$, but $M_{i}$ does not.
So far, we have simply introduced a sligtly different notation for
$M_i$, since the expectation operation ensures that the value of $M_i$
is unchanged.

We let $\xi \in \{0,1\}^m$ denote the binary vector indicating the
positions of indices $\gamma_1, \cdots, \gamma_i, {\gamma}_{i+1}, \cdots,$ ${\gamma}_k$.
while letting $\wh{\xi} \in \{0,1\}^m$ indicate the positions of
indices of $\gamma_1, \cdots, \gamma_i, \wh{\gamma}_{i+1}, \cdots, \wh{\gamma}_k$.




Note that $k$-homogeneous Strongly Rayleigh implies the stochastic
covering property. 
By Fact~\ref{fac:scp},
 the stochastic covering property
implies that a coupling exists s.t.\ adding either \emph{one} or
\emph{no} extra ``ones'' to the vector 
$\xi_{[m]\setminus \gamma_{i+1}}$  results in the vector
$\wh{\xi}_{[m]\setminus \gamma_{i+1}}$.
But, since $\xi_{[m]\setminus \gamma_{i+1}}$ is $k-1$ homogenous and
$\wh{\xi}_{[m]}$ is $k$-homogenous, we can conclude that 
$\wh{\xi}_{[m]\setminus \gamma_{i+1}}$ is obtained from
$\xi_{[m]\setminus \gamma_{i+1}}$ by adding \emph{no} ones,
if and only $\wh{\xi}_{[m]}$ has a one at index $\gamma_{i+1}$.
From this we conclude a more helpful form of stochastic covering: 
we can construct an index $\widetilde{\gamma}_{i+1}$ and a
coupling s.t. conditional on $\gamma_{i+1}$,
the indices
 $\widetilde{\gamma}_{i+1}, \gamma_{i+2}, \cdots, {\gamma}_k$
have the same distribution as $\wh{\gamma}_{i+1}, \wh{\gamma}_{i+2},
\cdots, \wh{\gamma}_k$.
Thus
\begin{align}\label{eq:the_equation_needs_coupling_table_exists}
  {Z}_{ \geq i+2 }  + Y_{\widetilde{\gamma}_{i+1}} =  \wh{Z}_{\geq i+1}.
\end{align}


We define $X_{i+1} = M_{i+1} - M_i$ and then by Eq.~\eqref{eq:mdiffzeromean}
\begin{align}\label{eq:X_i_is_zero_mean}
\E_{\gamma_{i+1}}[X_{i+1} ~|~ \gamma_1, \cdots, \gamma_i ] = 0 .
\end{align}
Then we can rewrite $X_{i+1}$ in the following way,
\begin{claim}\label{cla:rewrite_X_i_plus_1}
Let ${\cal D}_{\gamma_{\leq i+1}}$ denote the coupling distribution between ${Z}_{\geq i+2}$ and $Y_{ \wt{\gamma}_{i+1} }$ such that Eq.~\eqref{eq:the_equation_needs_coupling_table_exists} holds. Then 
\begin{align*}
X_{i+1} = Y_{{\gamma}_{i+1}} - \E_{ (  {Z}_{\geq i+2} , Y_{ \widetilde{\gamma}_{i+1} } ) \sim {\cal D}_{ \gamma_{\leq i+1}} } [ Y_{\widetilde{\gamma}_{i+1}} ~|~ \gamma_{\leq i+1 } ].
\end{align*} 
\end{claim}

\ifdefined\focsversion
We provide the proof of the above Claim in the full version.
\else
\begin{proof}
 Note that in the following proof, we should think of $\gamma_{\leq i+1 }$ as fixed.
\begin{align}\label{eq:rewrite_M_i_plus_1_minus_M_i}
X_{i+1}  
= & ~ M_{i+1} - M_i \notag \\
= & ~ \E_{\gamma_{\geq i+2}} \left[ Z ~|~ \gamma_{\leq i+1} \right] - \E_{\gamma_{\geq i+1}} \left[ Z ~|~  \gamma_{\leq i} \right] \notag \\
= & ~ \E_{  {\gamma}_{\geq i+2}} \left[  {Z}_{\geq i+2} + Y_{
    {\gamma}_{i+1}} + Z_{\leq i} ~\bigg|~ \gamma_{\leq i+1} \right] - \E_{\wh{\gamma}_{\geq i+1}} \left[
    \wh{Z}_{\geq i+1} + Z_{\leq i} ~\bigg|~  \gamma_{\leq i} \right]
    \notag \\
= & ~ \E_{ {\gamma}_{\geq i+2}} \left[  {Z}_{\geq i+2} + Y_{
    {\gamma}_{i+1}} + Z_{\leq i} ~\bigg|~ \gamma_{\leq i+1} \right] -
\E_{ (  {Z}_{\geq i+2} , Y_{ \widetilde{\gamma}_{i+1} } ) \sim {\cal D}_{\gamma_{\leq i+1}} } \left[
      {Z}_{ \geq i+2 }  + Y_{\widetilde{\gamma}_{i+1}}  + Z_{\leq i} ~\bigg|~  \gamma_{\leq i+1} \right]
\notag \\
= & ~ \E_{ (  {Z}_{\geq i+2} , Y_{ \widetilde{\gamma}_{i+1} } ) \sim {\cal D}_{\gamma_{\leq i+1}} } \left[ Y_{ {\gamma}_{i+1}} - Y_{\widetilde{\gamma}_{i+1}} ~|~ \gamma_{\leq i+1} \right]
\end{align}
where the first equality follows by definition of $X_{i+1}$, the
second equality follows by Definition~\ref{def:M}, the third equality
follows by Eq.~\eqref{eq:mdecomps},
the fourth equality follows by
Eq.~\eqref{eq:the_equation_needs_coupling_table_exists},
and the fifth equality is by linearity of expectation and cancellation of terms that agree.


Once we condition on $\gamma_{\leq i}$ and $ {\gamma}_{i+1}$ being
fixed, then $Y_{\gamma_{i+1}}$ is also fixed. Thus in
Eq.~\eqref{eq:rewrite_M_i_plus_1_minus_M_i}, we can move
$Y_{\gamma_{i+1}}$ out of Expectation, so that the right hand side of 
Eq.~\eqref{eq:rewrite_M_i_plus_1_minus_M_i} becomes
\begin{align}\label{eq:move_Y_gamma_i_plus_1_out_of_E}
\E_{ (  {Z}_{\geq i+2}, Y_{\widetilde{\gamma}_{i+1}} ) \sim {\cal D}_{\gamma_{\leq i+1}} } \left[ Y_{\gamma_{i+1}} - Y_{\widetilde{\gamma}_{i+1}} ~|~ \gamma_{ \leq i+1 }  \right] = Y_{{\gamma}_{i+1}} - \E_{ (  {Z}_{\geq i+2} , Y_{ \widetilde{\gamma}_{i+1} } ) \sim {\cal D}_{ \gamma_{\leq i+1}} } [ Y_{\widetilde{\gamma}_{i+1}} ~|~ \gamma_{\leq i+1 } ].
\end{align}

\Rasmus{Need to double check proposition 2.2 in \cite{pp14}. Need to cite something about why distribution ${\cal D}$ is actually existing}
\end{proof}
\fi

\begin{fact}\label{fac:four_properties_fact}
We condition on $\gamma_{\leq i+1}$. Let ${\cal D}_{\gamma_{\leq i+1}}$ denote the coupling distribution such that $ {Z}_{\geq i+2} + Y_{\widetilde{\gamma}_{i+1}} = \wh{Z}_{\geq i+1}$ holds. We define $U_{ {\gamma}_{i+1} }$ as follows 
\begin{align*}
U_{ {\gamma}_{i+1} } = \E_{ (  {Z}_{\geq i+2} , Y_{ \widetilde{\gamma}_{i+1} } ) \sim {\cal D}_{\gamma_{\leq i+1}} } [ Y_{\widetilde{\gamma}_{i+1}} ~|~ \gamma_{\leq i+1} ].
\end{align*}
Then, we have the following four properties,
\begin{align*}
\mathrm{(\RN{1})} & ~ \E_{ {\gamma}_{i+1}} [ U_{ {\gamma}_{i+1}} ~|~ \gamma_{\leq i} ] = \E_{\gamma_{i+1}} [ Y_{ {\gamma}_{i+1}} ~|~ \gamma_{\leq i} ], \\
\mathrm{(\RN{2})} & ~ \| Y_{ {\gamma}_{i+1}} \| \leq R, \| U_{ {\gamma}_{i+1}} \| \leq R,\\
\mathrm{(\RN{3})} & ~ \| Y_{\gamma_{i+1}} - U_{\gamma_{i+1}} \| \leq R, \\
\mathrm{(\RN{4})} & ~ Y_{ {\gamma}_{i+1}}^2 \preceq R \cdot Y_{ {\gamma}_{i+1}}, U_{ {\gamma}_{i+1}}^2 \preceq R \cdot U_{ {\gamma}_{i+1}}.
\end{align*}
\end{fact}

\ifdefined\focsversion
We provide the proof of the above Fact in the full version.
\else
\begin{proof}
Proof of (\RN{1}). We have
\begin{align*}
\E_{\gamma_{i+1}} [ U_{\gamma_{i+1}} ~|~ \gamma_{\leq i} ] 
= & ~ \E_{\gamma_{i+1}} \left[  \E_{ (  {Z}_{\geq i+2} , Y_{ \widetilde{\gamma}_{i+1} } ) \sim {\cal D}_{\gamma_{\leq i+1}} } \left[ Y_{\widetilde{\gamma}_{i+1}} ~\big|~ \gamma_{\leq i+1} \right] ~\bigg|~ \gamma_{\leq i} \right] \\
= & ~ \E_{\gamma_{i+1}} \left[  \E_{ (  {Z}_{\geq i+2} , Y_{ \widetilde{\gamma}_{i+1} } ) \sim {\cal D}_{\gamma_{\leq i+1}} } \left[ Y_{\widetilde{\gamma}_{i+1}} ~\big|~ \gamma_{\leq i+1} \right] - Y_{\gamma_{i+1}} + Y_{\gamma_{i+1}} ~\bigg|~ \gamma_{\leq i} \right] \\
= & ~ \E_{\gamma_{i+1}} \left[ -X_{i+1} + Y_{\gamma_{i+1}} ~\big|~ \gamma_{\leq i} \right] \\
= & ~ \E_{\gamma_{i+1}} \left[ -X_{i+1}  ~\big|~ \gamma_{\leq i} \right] + \E_{\gamma_{i+1}} \left[  Y_{\gamma_{i+1}} ~\big|~ \gamma_{\leq i} \right] \\
= & ~ \E_{\gamma_{i+1}} \left[  Y_{\gamma_{i+1}} ~\big|~ \gamma_{\leq i} \right],
\end{align*}
where the third step follows by
Eq.~\eqref{eq:move_Y_gamma_i_plus_1_out_of_E} and
Eq.~\eqref{eq:rewrite_M_i_plus_1_minus_M_i}, the fourth step follows
by linearity of expectation, and the last step follows by $\E_{\gamma_{i+1}} [ -X_{i+1} ~|~ \gamma_{\leq i} ] = 0$.

Proof of (\RN{2}).

By definition of $Y$, we have
$
\| Y_{ {\gamma}_{i+1}} \| \leq R.
$

\begin{align*}
\| U_{ {\gamma}_{i+1}} \| 
= & ~ \left\| \E_{ (  {Z}_{\geq i+2} , Y_{ \widetilde{\gamma}_{i+1} } ) \sim {\cal D}_{\gamma_{\leq i+1}} } [ Y_{\widetilde{\gamma}_{i+1}} ~|~ \gamma_{\leq i+1} ] \right\| \\
\leq & ~ \E_{ (  {Z}_{\geq i+2} , Y_{ \widetilde{\gamma}_{i+1} } )
       \sim {\cal D}_{\gamma_{\leq i+1}} } \left[ \left\|
       Y_{\widetilde{\gamma}_{i+1}} \right\|  ~|~ \gamma_{\leq i+1} \right] \\
\leq & ~ R.
\end{align*}

Proof of (\RN{3}). For any two PSD matrices $A$ and $B$, we have $\| A
- B \| \leq \max( \| A \|, \| B \| )$. Because both $Y_{\gamma_{i+1}}$
and $U_{\gamma_{i+1}}$ are PSD matrices and $ \max ( \|
Y_{\gamma_{i+1}} \|, \| U_{\gamma_{i+1}} \| ) \leq R$, we get the
desired property.

Proof of (\RN{4}).

It follows by (\RN{2}) and that $Y_{ {\gamma}_{i+1}}$ and $U_{
  {\gamma}_{i+1}}$ are both PSD matrices.
\end{proof}
\fi

We can show 
\begin{claim}\label{cla:square_of_Y_e_minus_U_e_is_at_most_4_Y_e}
\begin{align*}
\E_{  {\gamma}_{i+1} } \left[ ( Y_{ {\gamma}_{i+1}} - U_{ {\gamma}_{i+1}} )^2 ~|~ \gamma_{\leq i} \right] \preceq 4 R \cdot \E_{ {\gamma}_{i+1}}[ Y_{ {\gamma}_{i+1}} | \gamma_{\leq i}]
\end{align*}
\end{claim}
\begin{proof}
\begin{align*}
 & ~ \E_{  {\gamma}_{i+1} } \left[ ( Y_{ {\gamma}_{i+1}} - U_{ {\gamma}_{i+1}} )^2 ~|~ \gamma_{\leq i} \right] \\
\preceq & ~ \E_{  {\gamma}_{i+1} } \left[ 2 Y_{ {\gamma}_{i+1}}^2 + 2 U_{ {\gamma}_{i+1}}^2 ~|~ \gamma_{\leq i} \right] \\
\preceq & ~ \E_{  {\gamma}_{i+1} } \left[ 2 R \cdot Y_{ {\gamma}_{i+1}} + 2 R \cdot U_{ {\gamma}_{i+1}} ~|~ \gamma_{\leq i} \right] \\
\preceq & ~ \E_{  {\gamma}_{i+1} } \left[ 4 R \cdot Y_{ {\gamma}_{i+1}} ~|~ \gamma_{\leq i} \right],
\end{align*}
where the first step follows by Fact~\ref{fac:symmetric_matrix_triangle_inequality}, and the second step follows by $U_{ {\gamma}_{i+1}}^2 \preceq R \cdot U_{ {\gamma}_{i+1}}$ and $Y_{ {\gamma}_{i+1}}^2 \preceq R \cdot Y_{ {\gamma}_{i+1}}$.
\end{proof}

\begin{lemma}\label{lem:Y_e_is_at_most_1_over_n_minus_i}
Let $\E [ \sum_{e\in \Gamma} \xi_e Y_{e} ] \preceq \mu I$. 
For each $i \in \{1, 2, \cdots, k \}$, we have
\begin{align*}
\E_{ {\gamma}_i} \left[ Y_{ {\gamma}_i} ~|~ \gamma_{\leq i-1} \right] \preceq \frac{1}{k+1-i} \mu I.
\end{align*}
\end{lemma}
\begin{proof}
We use ${\bf 1}$ to denote a length $i-1$ vector where each entry is
one. We can think of $\gamma_{\leq i-1}$ as having its values already
set to some edges in $\Gamma$, for example $\gamma_1 = e_1, \cdots,
\gamma_{i-1} = e_{i-1}$.
Note that all of the $e_1, \cdots, e_{i-1}$
must be distinct. Then we use $\Gamma \backslash \gamma_{\leq i-1}$ to denote $\Gamma \backslash \{e_1, \cdots, e_{i-1}\}$.
\begin{align*}
\E \left[ Y_{ {\gamma}_i} ~|~ \gamma_{\leq i-1} \right] 
= & ~ \sum_{e \in \Gamma \backslash \gamma_{\leq i-1}} \Pr[ \gamma_i = e ~|~ \gamma_{\leq i-1} ] \cdot  Y_{e} \\
= & ~ \sum_{e \in \Gamma \backslash \gamma_{\leq i-1} } \frac{\Pr[ \xi_e = 1 ~|~ \gamma_{\leq i-1} ]  }{k-(i-1)} \cdot  Y_{e} \\
= & ~ \sum_{e \in \Gamma \backslash \gamma_{\leq i-1}} \frac{\Pr[ \xi_e = 1 ~|~ \xi_{\gamma_{\leq i-1}} = \bf{1} ]  }{k-(i-1)} \cdot  Y_{e} \\
\preceq & ~ \sum_{e \in \Gamma \backslash \gamma_{\leq i-1}} \frac{\Pr[ \xi_e = 1 ]  }{k-(i-1)} \cdot  Y_{e} \\
\preceq & ~ \sum_{e \in \Gamma} \frac{\Pr[ \xi_e = 1 ]  }{k-(i-1)} \cdot  Y_{e} \\
= & ~ \frac{1}{k-(i-1)} \E\left[\sum_{e} \xi_e Y_e \right] \\
\preceq & ~ \frac{1}{k+1 - i} \mu I
\end{align*}
where the first step follows by definition of expectation, the second
step follows by  $\Pr[ \gamma_i = e ~|~ \gamma_{\leq i-1} ]  =
\Pr[\xi_e = 1 ~|~ \gamma_{\leq i-1} ] / (k-(i-1))$, the third step
follows because $[\cdot | \gamma_{\leq i-1}]$ is equivalent to $[\cdot
| \xi_{\gamma \leq i-1} = {\bf 1} ]$, the fourth step follows by
($\Pr[\xi_e =1 ~|~ \xi_{\gamma \geq i-1} = {\bf 1}] \leq \Pr[\xi_e =
1]$) from the Shrinking Marginals Lemma~\ref{lem:intro_shrinking_marginals}, the fifth step follows by relaxing $\Gamma \backslash \gamma_{\leq i-1}$, the sixth step follows by $\Pr[\xi_e = 1] = \E[\xi_e]$ and linearity of expectation, and the last step follows by $\E[\sum_{e\in \Gamma} \xi_e Y_e ] \preceq \mu I$.

\end{proof}

\begin{lemma}\label{lem:E_e_i_X_i_2_is_at_most_4_over_n_i_I}
For each $i \in \{1,2,\cdots, k\}$
\begin{align*}
\E_{\gamma_i} \left[ X_i^2 ~|~ \gamma_{\leq i-1} \right] \preceq 4 \mu R \frac{1}{k + 1 - i} I.
\end{align*}
\end{lemma}
\begin{proof}
It follows by combining Claim~\ref{cla:square_of_Y_e_minus_U_e_is_at_most_4_Y_e} and Lemma~\ref{lem:Y_e_is_at_most_1_over_n_minus_i} directly.
\end{proof}
The above lemma implies this corollary directly
\begin{corollary}\label{cor:bound_sum_of_X_square}
\begin{align*}
\sum_{i=1}^k \E_{\gamma_i} \left[ X_i^2 ~|~ \gamma_{\leq i-1} \right] \preceq 10 \mu R \log k \cdot I.
\end{align*}
\end{corollary}

\subsection{Main result}

Before finally proving our main theorem~\ref{thm:srmatchernoff}, we state a useful tool: Freedman's inequality for matrices

We state a version from \cite{t11b}, and there is also another version can be found in \cite{o09}.
\begin{theorem}[Matrix Freedman]\label{thm:matrix_freedman}
Consider a matrix martingale $\{Y_i : i = 0, 1, 2, \cdots \}$ whose values are self-adjoint matrices with dimension $n$, and let $\{ X_i : i = 1,2,3,\cdots \}$ be the difference sequence. Assume that the difference sequence is uniformly bounded in the sense that
\begin{align*}
\lambda_{\max} (X_i) \leq R, \mathrm{~almost~surely~}  \mathrm{~for~} i = 1, 2, 3, \cdots .
\end{align*}
Define the predictable quadratic variation process of the martingale :
\begin{align*}
W_i = \sum_{j=1}^i \E_{j-1} [ X_j^2 ], \mathrm{~for~} i = 1, 2, 3, \cdots .
\end{align*}
Then, for all $t \geq 0$ and $\sigma^2 > 0$,
\begin{align*}
 & ~ \Pr \left[ \exists i \geq 0 : \lambda_{\max} (Y_i) \geq t \mathrm{~and~} \| W_i \| \leq \sigma^2 \right] \\
\leq & ~ n \cdot \exp \left( - \frac{t^2/2}{\sigma^2 + Rt/3} \right).
\end{align*}
\end{theorem}

Now, we are ready to prove our main theorem,
\restate{thm:srmatchernoff}

\begin{proof}
We use $Y$ to denote $A$ and $\Gamma$ to denote $[m]$.

In order to use Theorem~\ref{thm:matrix_freedman}, we first we define $W_i$ as follows
\begin{align*}
W_i = \sum_{j=1}^i \E_{\gamma_i} \left[ X_i^2 ~|~ \gamma_{\leq i-1} \right].
\end{align*}

According to definition of $M_i$, $\{ M_0, M_1, M_2 \cdots \}$ is a matrix martingale and $M_k - M_0 = \sum_{e} \xi_e A_e - \E[ \sum_{e} \xi_e A_e ]$.

We have proved the following facts,

The first one is, $\E_{\gamma_{i}} [ X_i | \gamma_{\leq i-1}] = 0$. It follows by Eq.~\eqref{eq:X_i_is_zero_mean}

The second one is 
\begin{align*}
\lambda_{\max}(X_i) \leq R
\end{align*}
It follows by combining Property (\RN{3}) of Fact~\ref{fac:four_properties_fact} and Claim~\ref{cla:rewrite_X_i_plus_1}.

The third one is
\begin{align*}
\| W_i \| \leq \sigma^2, \forall i \in [k]
\end{align*}
where $ \sigma^2 = 10\mu R \log k$. It follows by Corollary~\ref{cor:bound_sum_of_X_square}.

Thus, 
\begin{align*}
\Pr \left[ \lambda_{\max} ( M_k - M_0 ) \geq \epsilon \mu \right]  \leq n \exp \left( -\frac{ (\epsilon \mu)^2/2}{ \sigma^2 + R (\epsilon \mu)/3 } \right).
\end{align*}
We have
\begin{align*}
\frac{ t^2/2}{ \sigma^2 + R t/3 } = & ~ \frac{\epsilon^2 \mu^2 /2}{10\mu R \log k + R \epsilon \mu /3} & \text{~by~choosing~} t = \epsilon \mu \\
= & ~ \frac{ 3 \epsilon^2 \mu }{ (60 \log k + 2\epsilon ) R }. 
\end{align*}
Thus we prove one side of the bound. Since $\E_{\gamma_i} [ -X_i | \gamma_{\leq i-1}] = 0$ and $\E_{\gamma_i} [ (-X_i)^2 | \gamma_{\leq i-1} ] = \E_{\gamma_i} [ X_i^2 | \gamma_{\leq i-1} ]$, then following the similar procedure as proving $\lambda_{\max}$, we have bound for $\lambda_{\min}$ 
\begin{align*}
\Pr[ \lambda_{\min} (M_k - M_0) \leq -\epsilon \mu ] \leq n \exp \left( - \frac{ 3 \epsilon^2 \mu }{ (60 \log k + 2\epsilon ) R } \right).
\end{align*}
Putting two sides of the bound together, we complete the proof.
\end{proof}

\section{Applications to Random Spanning Trees}

In this section, we show how to use Theorem~\ref{thm:srmatchernoff} to prove the bound for one random spanning and also summation of random spanning trees.

\restate{thm:intro_one_spanning_tree_upper}

\begin{proof}

Let $G = (V,E,w)$ be a undirected weighted graph, $w : E \to R$, which
is connected. 
The Laplacian of $G$ is $L_G = \sum_{e \in E}  w(e) b_e b_e^\intercal$.

Let $T \subseteq E$ be a random spanning tree of $G$ in the sense of
Definition~\ref{def:wtUnifTreeDistr}.
Let the weights of the edges in $T$ be given by $w' : T \to R$ where 
$w'(e) = w(e) / l_e$, where $l_e$ is the leverage score of $e$ in $G$.
Thus the Laplacian of the tree is $L_T = \sum_{e \in T}  w'(e) b_e
b_e^\intercal = \sum_{e \in T}  \frac{w(e)}{l_e} b_e b_e^\intercal$.
Then by Fact~\ref{fac:spantreemargin}, $Pr[ e \in T ] = l_e$, and
hence $\E[ L_T ] = L_G$.

Note also that for all $e \in E$,  $|| (L_G^\dagger)^{1/2} w(e)b_e
b_e^\intercal (L_G^\dagger)^{1/2} || = l_e $.
Consider the random matrix $ (L_G^\dagger)^{1/2}  L_T
(L_G^\dagger)^{1/2}$.
The distribution of edge in the spanning tree can be seen as an $n-1$
homogeneous vector in $\{0,1\}^m$ where $m = |E|$.
To apply Theorem~\ref{thm:srmatchernoff}, 
let $\xi_e$ be the $e$th entry of this random vector, and
\[
A_e =
 (L_G^\dagger)^{1/2} w'(e)b_e
b_e^\intercal (L_G^\dagger)^{1/2}
\]
Note $A_e \succeq 0$.
Now $|| A_e || = 1$ and $\E\left[ \sum_{e} \xi_e
      A_e\right] = \E[(L_G^\dagger)^{1/2} L_T (L_G^\dagger)^{1/2} ] =
    (L_G^\dagger)^{1/2} L_G (L_G^\dagger)^{1/2} = \Pi = I -
    \frac{1}{n} \mathbf{1}
    \mathbf{1}^\intercal$, where we used in the last equality that the
    null space of the Laplacian of a connected graph is the span of
    the all ones vector.
Thus, as each we get $||\E\left[ \sum_{e} \xi_e A_e\right]|| = 1$
This means we can apply Theorem~\ref{thm:srmatchernoff} with $R = 1$,
$\mu = 1$ and $\epsilon = 100 \log n $ to whp.
$ || (L_G^\dagger)^{1/2} L_T (L_G^\dagger)^{1/2} - \Pi || \leq 100 \log
n $.

As $L_T$ is a Laplacian, it has $\mathbf{1}$ in the null space, so can
conclude that $(L_G^\dagger)^{1/2} L_T (L_G^\dagger)^{1/2} \preceq 100
\log n \Pi$. Hence $ L_T \preceq \log n L_G$.

\end{proof}

\restate{thm:intro_sum_spanning_trees}

\begin{proof}
  The proof is similar to the proof of
  Theorem~\ref{thm:intro_one_spanning_tree_upper}.
  Now we view the edges of $ t = O(\epsilon^{-2} \log^2 n)$ 
  independent random spanning trees as a $t (n-1)$-homogeneous
  Strongly Rayleigh Distribution a vector in $\{0,1\}^{t |E|}$ .
  Note that the product of independent Strongly Rayleigh distributions
  is Strongly Rayleigh \cite{bbl09}.
  Again we get $||\E\left[ \sum_{e} \xi_e A_e\right]|| = 1$, but now
  we can take $R = \frac{1}{t}$, and hence we obtain the desired result
  by plugging into Theorem~\ref{thm:srmatchernoff}.
\end{proof}


\section{Lower bounds}

\subsection{Single spanning tree, low probability}

The goal of this section is to prove Theorem~\ref{thm:intro_one_spanning_tree_lower_low_probability}.
First, we recall a helpful fact estbliashed by Pr\"{u}fer \cite{p18}.
\begin{fact}
\label{fac:treedeg}
  If $T$ is a uniformly random spanning tree of the complete graph $G$
  on $n$ vertices, the degree distribution of a fixed node $v$ in $T$ is 
  $1 + \text{Binomial}(n-2,1/n)$. 
\end{fact}
We now prove two claims that will serve as helpful tools,
Claims~\ref{cla:probability_that_degree_of_a_fixed_node} and~\ref{cla:bad_approximation_ratio_for_one_graph}.
\begin{claim}\label{cla:probability_that_degree_of_a_fixed_node}
Let $G$ be complete graph $K_{n}$ with $n \geq 4$, let $T$ denote a random spanning tree, the probability that at least one node of the $T$ has degree at least $b \log n$ is at least $2^{- b \log n \log ( b \log n ) - 3}$.
\end{claim}
\begin{proof}
By Fact~\ref{fac:treedeg}, the degree distribution of a fixed node in $T$ is, $1 + \text{Binomial}(n-2,1/n)$. 

For a random variable $x$ sampled from $\text{Binomial}(n-2,1/n)$, we use $q_i$ to denote the probability that $x = i$.

Let $p = 1/n$. We consider $q_{b \log n}$, which is
\begin{align*}
q_{b \log n} = & ~ {n-2 \choose b \log n} \cdot p^{b \log n} \cdot (1-p)^{n -2 - b\log n} \\
= & ~ {n-2 \choose b \log n} \cdot ( 1/n )^{b \log n} \cdot (1- 1/n)^{n - 2 - b \log n} & \text{~by~} p = 1/n \\
\geq & ~ ( (n - 2) / (b \log n) )^{ b \log n} \cdot (1/n)^{b \log n} \cdot (1 - 1/n)^{n - 2 - b \log n} \\
= & ~ (b \log n)^{- b \log n} \cdot ( (n-2) / n )^{b \log n} \cdot (1- 1/n)^{ n - 2 - b \log n} \\
\geq & ~ (b \log n)^{- b \log n} \cdot (1  - 2/n )^{b \log n} \cdot (1-2/n)^{ n - 2 - b \log n } \\
= & ~ (b \log n)^{- b \log n} \cdot (1-2/n)^{n-2} \\
\geq & ~ (b \log n)^{- b \log n} \cdot \frac{1}{e^2}\\
\geq & ~ 2^{- b \log n \cdot \log (b \log n ) - 3}.
\end{align*}
where the seventh step follows by $(1-2/n)^{n-2} \geq 1/e^2$ when $n \geq 4$.
Then the desired probability is
\begin{align*}
\sum_{i= b \log n}^{n} q_i \geq 2^{- b \log n \cdot \log (b \log n ) - 3}.
\end{align*}
\end{proof}

\begin{claim}\label{cla:bad_approximation_ratio_for_one_graph}
Let $G$ be a complete graph $K_{n}$, let $T$ denote a random spanning
tree, if $T$ has a node with degree at least $d$, then 
the inverse leverage score weighted Laplacian of the tree satisfies
\begin{align*}
L_T \not\preceq (d/2) \cdot L_G. 
\end{align*}
\end{claim}
\begin{proof}
There are $n(n-1)/2$ edges in the graph $G$. Let $l_e$ denote the
leverage of each $e \in G$. The the sum of the leverage scores is
$\sum_{e \in G} l_e = n-1$, e.g. see \cite{ss11}. Since all the edges
in the graph $G$ are symmetric, we have $l_e = 2/n$ for all edge $e$ in $G$.

Let $v$ denote a fixed node in graph $G$ and let $d$ be the degree of
$v$. Let $L_v$ denote the Laplacian matrix of the subgraph of $T$
consisting of edges incident on $v$, i.e. the star of $d+1$ nodes with
$v$ at the center, and with edge weights as in $T$ (which differ from
those in $G$).
We should think of $L_v$ as a $n \times n$ matrix with only $d+1$
nonzeros on the diagonal.

Observe that $\lambda_{\max}(L_v) \leq \frac{n}{2}\lambda_{\max}(L_{K_{d+1}})
\leq \frac{n}{2} (d+1)$.
We can also exhibit a unit vector $x = \frac{1}{\sqrt{d^2 + d}}
(d,-1,\cdots,-1) $, for which $x^\top L_v x = \frac{n}{2} (d+1)$ which implies that
$\lambda_{\max}(L_v) \geq \frac{n}{2} (d+1)$.
Therefore $\lambda_{\max}(L_v) =\frac{n}{2} (d+1)$.

We can split the $L_T$ into two parts,
\begin{align*}
L_T = L_v + L_{T\backslash v}
\end{align*}
and both parts are PSD matrices.
We also know that $\lambda_{\max}(L_G) = n$.
Thus,
\begin{align*}
L_T \not\preceq (d/2) \cdot L_G.
\end{align*}
\end{proof}

\restate{thm:intro_one_spanning_tree_lower_low_probability}
\begin{proof}
The proof is a direct combination of Claim~\ref{cla:probability_that_degree_of_a_fixed_node} and Claim~\ref{cla:bad_approximation_ratio_for_one_graph}.

The approximation ratio is
\begin{align*}
\frac{d}{2} = \frac{ b \log n }{ 2 } = 10 \log n,
\end{align*}
where the last step follows by choosing $b = 20$.

Then the probability is
\begin{align*}
2^{- b \log n \log (b \log n) - 3} = & ~ 2^{ - b \log n \log\log n - b \log n \log b - 3  } \\
\geq & ~ 2^{ - 20 \log n \log\log n - 20 \log n \log 20 - 3 } \\
\geq & ~ 2^{-150 \log n \log \log n}
\end{align*}
where the last step follows by $\log \log n \geq 1$ when $n \geq 4$.
\end{proof}

\subsection{Single spanning tree, high probability}

\restate{thm:intro_one_spanning_tree_lower_high_probability}

\begin{proof}

Let $C = 4$, and let $\delta = 1/(C\log\log n)$.
Note we have assumed $n \geq 2^{2^6}$, which ensures that $\delta < 0.05$.

We constuct a graph of size $n$ as a union of $n^{1-\delta}$
cliques of size $n^{1-\delta}$ that are disjoint except they all share
one central vertex.
Applying Claim~\ref{cla:probability_that_degree_of_a_fixed_node} with
$n$ replaced by $n^{\delta}$, and $b=1$, and assuming
$n^\delta \geq 8$, we get that 
for each clique, the probability that at least one node has degree at least $\log n^{\delta}$ in $T$ is at least
\begin{align*}
2^{-\delta c_0 \log n \cdot \log (\delta \log n)} 
\end{align*}
where $c_0 = 2$.
We can lower bound this probability:
\begin{align*}
2^{-\delta c_0 \log n \cdot \log (\delta \log n)} = & ~ 2^{ - \frac{c_0 \log n}{C\log \log n} \log ( \frac{1}{C\log\log n} \log n ) } \\
\geq & ~ 2^{ - \frac{c_0 \log n}{C\log \log n} \log \log n } \\
= & ~ 2^{- (c_0 \log n) / C} \\
= & ~ n^{-c_0/C},
\end{align*}
where the first step follows by $\delta = 1/(C\log\log n)$.

The probability that at least one node in $G$ has degree at least $\log n^{\delta}$ in $T$ is at least 
\begin{align*}
1- ( 1 - 2^{-\delta c_0 \log n  \log (\delta \log n) } )^{n^{1-\delta}} 
= & ~ 1 - \left(1 - \frac{1}{n^{c_0/C}}\right)^{ n^{1-\delta} } \\
= & ~ 1 - \left(1 - \frac{1}{n^{c_0/C}}\right)^{ n^{c_0 /C} \cdot \frac{n^{1-\delta}}{n^{c_0/C}} } \\
\geq & ~ 1 - (1/e)^{n^{1-\delta-{c_0/C}}} \\
\geq & ~ 1 - (1/e)^{n^{0.4}},
\end{align*}
where the last step follows by $\delta \leq 0.05$, $c_0 = 2$ and $C =
4$.
Thus, we have the desired probability.

Using Claim~\ref{cla:bad_approximation_ratio_for_one_graph}, we have the approximation ratio 
\begin{align*}
\frac{d}{2} = \frac{ \log (n^\delta) }{2} = \frac{ \delta \log n }{2} = \frac{ \log n }{ 2 C \log \log n} = \frac{\log n}{8 \log \log n}.
\end{align*}

Note that we still need to make sure $n^{\delta} \geq 8$, which is
implied by $2^{\log n / 4 \log \log n} \geq 2^3$ which is equivalent
to $\log n \geq 12 \log \log n$.
This holds for all $n\geq 2^{2^6}$:
At $n=2^{2^6}$, $\log n = 2^6 > 60 = 12 \cdot 5 = 12 \log \log n$, and
as $n$ grows, the left hand side grows faster than the right hand side.

\end{proof}

\subsection{Sum of a batch of spanning trees}

\restate{thm:intro_multiple_spanning_tree_lower_high_probability}

\begin{proof}

We constuct a graph of size $n$ as a union of $n^{1-\delta}$
cliques of size $n^{1-\delta}$ that are disjoint except they all share
one central vertex.
The parameter $\delta \in (0,1)$ will be decided later.

We use $H$ to denote the graph formed by a collection of trees $T_1, T_2, \cdots, T_t$. 
Let $L_H$ denote the Laplacian matrix of new graph $H$. 

We use $\deg(v)$ to denote the degree of a vertex $v$.
We use $\wdeg(v)$ to denote weighted degree (after re-weighting).
In the original graph $G$ and the new graph $H$,
we have for each vertex $v$, that 
\begin{align*}
\wdeg_H(v) = \frac{d}{2 t} \deg_H(v), \text{~and~} \wdeg_G(v) = \deg_G(v).
\end{align*}
By our construction of the graph, it is easy to see that $\deg_G(v) =
n^\delta$ for all vertices $v$ except the special central vertex that
appears in all the cliques.
Let $d=n^\delta$.
Let $\xi_1$ denote 
the event that there exists a vertex $x \in V$ such that
\begin{align*}
\wdeg_H(x) > (1+\epsilon) \wdeg_G(x),
\end{align*}
and let $\xi_2$ denote that
there exists a vertex $y\in V$, such that
\begin{align*}
\wdeg_H(y) < (1-\epsilon) \wdeg_G(y).
\end{align*}
We want to show that events $\xi_1$ and $\xi_2$ both occur
simultaneously with probability at least $1-e^{-n^{0.39}}$, which
implies absence of spectral $(1\pm\epsilon)$ approximation, as desired.
We first bound the probability of $\xi_1$.
Note that
\begin{align*}
\Pr [ \wdeg_H(v) \geq (1+\epsilon) \wdeg_G(v) ] = & ~ \Pr \left[ \frac{ d }{ 2 t } \deg_H(v) \geq (1+\epsilon) \deg_G(v) \right] \\
\geq & ~ \Pr \left[ \frac{ d }{ 2 t } \deg_H(v) \geq (1+\epsilon) d \right] \\
= & ~ \Pr[ \deg_H(v) \geq 2 t (1+\epsilon) ] \\
= & ~ \Pr[ \deg_H(v) - t \geq t( 1 + 2\epsilon ) ].
\end{align*}

By Fact~\ref{fac:treedeg}, the degree of a fixed node in $T_i$ is
distributed as $1+\text{Binomial}(d-2,1/d)$. Then as $H$ is a union of
independent spanning trees, the degree in $H$ of a fixed node is
distributed as $t+ \text{Binomial}(t(d-2), 1/d)$.

For a random variable $x$ sampled from $\text{Binomial}(t(d-2),1/d)$,
we know that $\E[x]=t(1-2/d)$.
For $\epsilon > 5/d$
\begin{align*}
t( 1 + 2\epsilon ) 
= & ~ t(1-2/d) + (2 t /d - 2 \epsilon t + 8 \epsilon t /d ) + 4 \epsilon t (1-2/d) \\ 
\leq & ~ t(1-2/d) + (10t /d - 2\epsilon t) + 4 \epsilon t (1-2/d) & \text{~by~}\epsilon \leq 1\\
\leq & ~ t (1-2/d) + 4\epsilon t (1-2/d) & \text{~by~} \epsilon > 5 /d.
\end{align*}
So it suffices to calculate the probability that 
\begin{align}\label{eq:x_geq_1+eps_t}
x \geq t (1-2/d) + 4\epsilon t (1-2/d).
\end{align}

For any $k \geq 10, p \in (0,1/2), \epsilon \in (0,1/2)$ with $\epsilon^2 p k \geq 3$,
using Lemma~\ref{lem:reverse_chernoff_bound}, 
we can prove the probability that Eq.~\eqref{eq:x_geq_1+eps_t} holds is at least $2^{-c \epsilon^2 k p}$, 
where $c=9$. 
We choose $k = t(d-2)$ and $p=1/d$, and get that this probability
is at least $2^{-c \epsilon^2 t(d-2)/d}$.
Now, the probability that event $\xi_1$ holds is at least
\begin{align*}
1 - (1 - 2^{-c \epsilon^2  t (1-2/d)} )^{n^{1-\delta}}
\end{align*}
We have
\begin{align*}
2^{-c \epsilon^2  t (1-2/d)} 
\geq & ~ 2^{-c \epsilon^2 t } \\
\geq & ~ 1 / n^{0.5}
\end{align*}
where the last step follows by $t \leq 0.5 \log n/(c\epsilon^2)$.

Thus, we have
\begin{align*}
1 - (1 - 2^{-c \epsilon^2  t (1-2/d)} )^{n^{1-\delta}} 
\geq & ~ 1 - (1 - 1/n^{0.5})^{n^{1-\delta}} \\
\geq & ~ 1 - e^{-n^{1-\delta-0.5}} \\
= & ~ 1 - e^{-n^{0.4}} & \text{~by~} \delta = 0.1
\end{align*}

We summarize the conditions for $\epsilon$:
\begin{align*}
\epsilon \geq & ~ \max(5/d, 3 / \sqrt{pn}) \\
= & ~ \max ( 5 / n^{0.1}, 3 /\sqrt{pn} ) & \text{~by~} d= n^{0.1} \\
\geq & ~ \max (5/n^{0.1}, 5 /\sqrt{t})
\end{align*}
Since we choose $t = 0.05 \epsilon^{-2} \log n $, then as long as $\log n \geq 100$ we have $\epsilon \geq 5 / \sqrt{t}$.

Similarly, we can control the probability of event $\xi_2$ similarly,
completing the proof.

\end{proof}


\section{Shrinking Marginals Lemma}\label{sec:marginals}

\begin{lemma}[Restatement of Lemma~\ref{lem:intro_shrinking_marginals}, Shrinking Marginals]
Suppose $ (\xi_1, \ldots, \xi_m) \in  \{ 0,1\}^m$ is a
random vector of $\{0,1\}$ variables whose distribution is
$k$-homogeneous and Strongly Rayleigh,
then any set $S \subseteq [m]$ with $|S| \leq k$
for all $j \in [m]\setminus S$
\[
\Pr[ \xi_j = 1 | \xi_S = \mathbf{1}_S]
\leq
\Pr[ \xi_j = 1 ] 
\]
\end{lemma}

\begin{proof}
Note that by an immediate consequence of negative association, for any pair $i,j \in [m]$, with $i \neq j$,
$$\Pr[ \xi_j = 1 ~|~ \xi_i = 1 ] \leq \Pr[ \xi_j = 1 ~|~ \xi_i  = 0 ].$$
Hence 
\begin{align*}
\Pr[ \xi_j = 1 ~|~ \xi_i = 1 ] 
\leq & ~  \Pr[ \xi_j = 1 | \xi_i = 1 ] \cdot \Pr[\xi_i = 1 ] + \Pr[ \xi_j = 1 | \xi_i = 0 ] \cdot (1 -  \Pr[\xi_i = 1 ]) \\
=    & ~ \Pr[ \xi_j = 1 ]
\end{align*}
By \cite{bbl09}, the distribution of $\xi_{[m] \backslash \{i\}} \in  \{ 0,1\}^{m-1}$
conditional on $\xi_i = 1$ is Strongly Rayleigh.

With loss of generality, let us order the indices s.t. $S = \{ 1, \ldots, s\}$, where $s
\leq k$. We use $[i]$ to denote $\{1,2, \cdots, i\}$.
Using the above observations, we can now prove the lemma by induction.
The induction hypothesis at the $i$-th step  (where $i \leq k$), is
that the following two statements are true.
\begin{enumerate}
\item $\forall j \in \{i+1, \ldots, m\}. \Pr[ \xi_j = 1 ~|~ \xi_{[i]} = \mathbf{1} ] \leq \Pr[ \xi_j = 1] $
\item The distribution of the vector of random variables $\xi_{m \backslash [i] }
  \in  \{ 0,1\}^{m-i}$ conditional on $\xi_{[i]} = \mathbf{1}$ is Strongly Rayleigh.
\end{enumerate}



\end{proof}


\section*{Acknowledgements}
We thank Yin Tat Lee, Jelani Nelson, Daniel Spielman, Aviad
Rubinstein, and Zhengyu Wang for helpful discussions regarding lower
bound examples.
We also thank Neil Olver for pointing out that our current lower
bounds cannot separate the concentration behavior in the Strongly
Rayleigh case from the independent case, correcting a remark in an
earlier version of this paper. We thank Zhengyu Wang for proof-reading the proof of lower bound and provide some useful comments for presentation.

We want to acknowledge Michael Cohen, who recently passed away.
In personal communication, Michael told Rasmus that he could prove a
concentration result for averaging of spanning trees, although losing
several log factors.
Unfortunately, Michael did not relate the proof to anyone, but did say
it was not based on Strongly Rayleigh distributions.
We hope he would appreciate our approach.

We also want to acknowledge the Simons program ``Bridging Continuous
and Discrete Optimization'' and Nikhil Srivastava whose talks during the
program inspired us to study Strongly Rayleigh questions.

This work was done while Rasmus was a postdoc at Harvard,
and Zhao was a visiting student at Harvard, both hosted by Jelani Nelson.
Rasmus was supported by ONR grant N00014-17-1-2127.

\newpage
\addcontentsline{toc}{section}{References}
\bibliographystyle{alpha}
\bibliography{ref,ref-dkprs}
\newpage
\appendix
\section{Omitted Proofs}\label{sec:appendix}

\begin{fact}\label{fac:two_matrices_sum}
For any two square matrices $A$ and $B$, we have
\begin{align*}
\frac{1}{2} (A+B)^2 + \frac{1}{2} (A-B)^2 = A^2 + B^2
\end{align*}
\end{fact}
\begin{proof}
\begin{align*}
 \frac{1}{2} (A+B)^2 + \frac{1}{2} (A-B)^2 
 = & ~ \frac{1}{2} (A + B) (A +B) +  \frac{1}{2} (A-B) (A-B) \\
 = & ~ \frac{1}{2} (A^2 + BA + AB + B^2) + \frac{1}{2} (A^2 - BA - AB + B^2) \\
 = & ~ A^2 + B^2,
\end{align*}
which completes the proof.
\end{proof}

\begin{fact}\label{fac:symmetric_matrix_triangle_inequality}
For any two symmetric matrices
\begin{align*}
(A - B)^2 \preceq 2A^2 + 2B^2.
\end{align*}
\end{fact}
\begin{proof}
Using Fact~\ref{fac:two_matrices_sum}, we have
\begin{align*}
(A+ B)^2 + (A - B)^2 = 2A^2 + 2 B^2.
\end{align*}
Because $A$ and $B$ are symmetric matrices, then $(A+B)^2 \succeq 0$. It implies that
\begin{align*}
(A - B)^2 \preceq 2A^2 + 2 B^2,
\end{align*}
which completes the proof.
\end{proof}

\subsection{Reverse Chernoff bound}

In this Section, we prove that the classical Chernoff bound is tight in some regimes. There are several different proofs, e.g. \cite{m10,slud77,ab09,y12}. For completeness, we provide a proof from \cite{y12}.

\begin{fact}\label{fac:stiring_application}
If $1 \leq l \leq k-1$, then 
\begin{align*}
{ k \choose l } \geq \frac{1}{ e \sqrt{ 2\pi l } } ( \frac{k}{l} )^l ( \frac{k}{k - l} )^{k-l}.
\end{align*}
\end{fact}
\begin{proof}
By Stirling's approximation, $i! = \sqrt{ 2\pi i} (i/e)^i e^{\lambda}$
for some $\lambda \in [ 1 / ( 12 i + 1 ) , 1/ ( 12 i ) ]$.

Thus, 
\begin{align*}
{k \choose l} 
= & ~ \frac{ k ! }{ l ! ( k - l ) ! } \\
\geq & ~ \frac{ \sqrt{2\pi k } ( \frac{k}{e} )^k }{ \sqrt{2\pi l} ( \frac{l}{e} )^l \sqrt{ 2\pi (k - l) } ( \frac{ k - l }{e} )^{k-l} } \exp( \frac{1}{12k+1} -\frac{1}{12 l} - \frac{1}{12 (k-l)} ) \\
\geq & ~ \frac{ \sqrt{2\pi k } ( \frac{k}{e} )^k }{ \sqrt{2\pi l} ( \frac{l}{e} )^l \sqrt{ 2\pi (k - l) } ( \frac{ k - l }{e} )^{k-l} } e^{-1} \\
\geq & ~ \frac{1}{ \sqrt{2 \pi l} } ( \frac{k}{l} )^l ( \frac{ k }{ k - l } )^{k - l} e^{-1} .
\end{align*}
where the first step follows by definition, the second step follows by Stirling's approximation, the third step follows by $\frac{1}{ 1 + a + b } + 1 \geq \frac{1}{a} + \frac{1}{b}$ for $a\geq 12, b \geq 12$.

\end{proof}

Now, we are ready to prove the following result
\begin{lemma}[\cite{m10,slud77,ab09,y12}]\label{lem:reverse_chernoff_bound}
Let $X$ be the average of $k$ independent, Bernoulli random variables with mean $p$. For any $\epsilon \in (0,1/2]$ and $p \in (0,1/2]$, assuming $\epsilon^2 p k \geq 3$, we have
\begin{align*}
\Pr[ X \leq ( 1 - \epsilon ) p ] \geq \exp( - 9 \epsilon^2 p k ),
\end{align*}
and
\begin{align*} 
\Pr[ X \geq ( 1 + \epsilon ) p ] \geq \exp( - 9 \epsilon^2 p k ).
\end{align*}
\end{lemma}
\begin{proof}
Note that $\Pr[ X \leq ( 1 - \epsilon ) p ]$ equals the sum $\sum_{i=0}^{ \lfloor ( 1 - \epsilon ) p k \rfloor } \Pr[ X = i / k ]$, and $\Pr[ X = i / k ] = { k \choose i } p^i ( 1 - p )^{k - i}$.

Fix $l = \lfloor ( 1 - 2 \epsilon ) p k \rfloor + 1$. The terms in the sum are increasing, so the terms with index $i\geq \ell$ each have value at least $\Pr[ X = l / k ]$, so their sum has total value at least $( \epsilon p k - 2 ) \Pr[ X = l /k ] $. To complete the proof, we show that
\begin{align*}
( \epsilon p k - 2 ) \Pr [ X = l / k ] \geq \exp ( - 9 \epsilon^2 p k ).
\end{align*}

The assumptions $\epsilon^2 p k \geq 3$ and $ \epsilon \leq 1 / 2 $ give $\epsilon p k \geq 6$, so we have
\begin{align*}
( \epsilon p k - 2 ) \Pr [ X = l / k ] \geq & ~ \frac{2}{3} \epsilon p k { k \choose l } p^l ( 1 - p )^{k - l} \\
\geq & ~ \underbrace{ \frac{2}{3} \epsilon p k \frac{1}{ \sqrt{ 2 \pi l } } }_{ A } \cdot \underbrace{ ( \frac{k}{l} )^l ( \frac{ k }{ k - l } )^{k - l} p^{ l } ( 1 - p )^{k - l} }_{ B } 
\end{align*}
Below, we will show that $A \geq \exp( - \epsilon^2 p k )$ and $B \geq \exp( - 8 \epsilon^2 p k )$.

\begin{claim}
$A \geq \exp( - \epsilon^2 p k )$.
\end{claim}
\begin{proof}
The assumptions $\epsilon^2 p k \geq 3$ and $\epsilon \leq 1/2$ imply that $pk \geq 12$.

By $l \leq pk + 1$ (from definition), and $pk \geq 12$, thus $l \leq 1.1 pk$.

Therefore, we have
\begin{align*}
A \geq & ~ \frac{2}{3e} \epsilon \sqrt{ p k / (2.2\pi) } \\
\geq & ~ \frac{2}{3e} \sqrt{ 3 / (2.2\pi) } & \text{~by~} \epsilon \sqrt{pk} \geq \sqrt{3} \\
\geq & ~ 0.1 \\
\geq & ~ \exp(-3) \\
\geq & ~ \exp( - \epsilon^2 p k ).
\end{align*}
This completes the proof of the claim.
\end{proof}

\begin{claim}
$B \geq \exp( - 8 \epsilon^2 p k )$.
\end{claim}
\begin{proof}
Fix $\delta$ such that $ l = ( 1 - \delta ) p k$. The choice of $l$
implies $\delta \leq 2 \epsilon$, so the claim will hold as long as $B
\geq \exp( - 2 \delta^2 pk )$.
Manipulting this latter inequality, we get
\begin{align*}
& ~ B^{ - 1 / l } \leq \exp( \frac{ 2 \delta^2 p k }{ l } ) \\
\iff & ~ \frac{ l }{ pk } ( \frac{ k - l }{ ( 1 - p ) k } )^{k / l - 1} \leq \exp( \frac{ 2 \delta^2 pk }{ l } ).
\end{align*}
Substituting $l = ( 1 - \delta ) p k$ and simplifying, it is equivalent to 
\begin{align*}
(1- \delta ) ( 1 + \frac{ \delta p }{ 1 - p } )^{ \frac{ 1 }{ ( 1 - \delta ) p } - 1 } \leq \exp( \frac{ 2 \delta^2 }{ 1 - \delta } ).
\end{align*}
Taking the logarithm of both sides, we have
\begin{align*}
\ln ( 1 - \delta ) + ( \frac{ 1 }{ ( 1 - \delta ) p } - 1 ) (  \ln  ( 1 + \frac{ \delta p }{ 1 - p } ) ) \leq \frac{ 2 \delta^2 }{ 1 - \delta }
\end{align*}
Since $\ln (1+z) \leq z$, it suffices to prove
\begin{align*}
& ~ - \delta + ( \frac{ 1 }{ ( 1 - \delta ) p } - 1 ) (  \frac{ \delta p }{ 1 - p } ) \leq \frac{ 2 \delta^2 }{ 1 - \delta } \\
\iff & ~ \frac{ \delta^2 }{ ( 1 - p ) ( 1 - \delta ) } \leq \frac{2 \delta^2 }{1 -\delta}.
\end{align*}
Since $p\leq 1/2$, this finishes the proof of the claim.
\end{proof}
Combining the above two claims, we obtain the desired probability lower bound for
the $(1-\epsilon)$ side. We can prove it for the $(1+\epsilon)$ side similarly.
\end{proof}



\end{document}